\documentclass[12pt]{amsart}
\usepackage{amssymb,amsmath,amsthm,newlfont,enumerate,graphicx, xcolor, comment}
\usepackage{etoolbox} 

\title[Infinitely supported $\mathcal{D}_\mu$ spaces which are $\mathcal{H}(b)$]{Infinitely supported harmonically weighted Dirichlet spaces which are de Branges Rovnyak spaces}

\author{C. Bellavita} 
\email{carlobellavita@ub.edu}
\address{Departament of Matematica i Informatica, Universitat de Barcelona, Gran Via 585, 08007 Barcelona, Spain.}

\author{E. Dellepiane} 
\email{eugenio.dellepiane@polito.it}
\address{Dipartimento di Scienze Matematiche “Giuseppe Luigi Lagrange”, Politecnico di Torino, Corso Duca degli Abruzzi 24, 10129 Torino, Italy}

\author{A. Hartmann} 
\email{andreas.hartmann@math.u-bordeaux.fr}
\address{Université de Bordeaux, CNRS, Bordeaux INP, IMB, UMR 5251, F-33400 Talence, France}

\author{J. Mashreghi} 
\email{javad.mashreghi@mat.ulaval.ca}
\address{D\'epartement de math\'ematiques et de statistique, Universit\'e Laval, Qu\'ebec, QC, Canada G1V 0A4.}


\date{\today}

\newtheorem{theorem}{Theorem}[section]
\newtheorem{corollary}[theorem]{Corollary}
\newtheorem{lemma}[theorem]{Lemma}
\newtheorem{proposition}[theorem]{Proposition}
\newtheorem{open}[theorem]{Open problem}
\numberwithin{equation}{section}

\theoremstyle{remark}

\theoremstyle{remark}

\makeatletter

\@addtoreset{equation}{section}
\@addtoreset{theorem}{section}

\makeatother

\let\oldequation\equation
\let\endoldequation\endequation
\renewenvironment{equation}
  {\refstepcounter{theorem}\oldequation}
  {\endoldequation}

\pretocmd{\align}{\refstepcounter{theorem}}{}{}
\pretocmd{\eqnarray}{\refstepcounter{theorem}}{}{}

\begin{document}

\begin{abstract}
Harmonically weighted Dirichlet spaces $\mathcal{D}_\mu$ and de Branges--Rovnyak spaces $\mathcal{H}(b)$ are two fundamental structures in analytic function theory exhibiting rich and often complementary properties. The question of when these spaces coincide, first raised and solved in Sarason’s groundbreaking work in 1997 when $\mu$ is a single Dirac mass, is thus of fundamental importance in operator theory and analytic function spaces. In this paper, we focus on spaces $\mathcal{H}(b)$ with symbol $b = (1+u)/2$, where $u$ is a one-component inner function. While previous results extended Sarason's work to finitely supported measures $\mu$, the symbols we consider here give a natural framework to go beyond finiteness of the support. In our setting, we provide a complete characterization of measures $\mu$ for which $\mathcal{H}(b) = \mathcal{D}_\mu$, thereby resolving the long-standing open problem of constructing harmonically weighted Dirichlet spaces $\mathcal{D}_\mu$ associated with measures $\mu$ of infinite support that are also $\mathcal{H}(b)$ spaces.
As a central ingredient to prove this result and which is of independent interest, we establish 
a $T(1)$-type result for the Cauchy transform on $L^2(\sigma)$, where $\sigma$ denotes the Clark measure associated with a one-component inner function $u$. Another notable result is a perturbation theorem for one-component inner functions that allows us to present a large class of function spaces satisfying $\mathcal{H}(b)=\mathcal{D}_\mu$. Furthermore, we settle the Brown--Shields conjecture within this setting.

\end{abstract}

\maketitle

\section{Introduction}
In this article, we investigate two important classes of Hilbert spaces of analytic functions on the unit disk $\mathbb{D}$: the de Branges--Rovnyak spaces $\mathcal{H}(b)$ and the harmonically weighted Dirichlet spaces $\mathcal{D}_\mu$. The former class, introduced by de Branges and Rovnyak in 1966 \cite{Debrangesrovnyak}, generalizes the orthogonal complement of the range of the multiplication operator $M_b$. The latter was introduced by Richter in 1991 \cite{Richter1991} in the context of his characterization of analytic cyclic 2-isometries. The weighted Dirichlet spaces also play a central role in the study of closed subspaces of the classical Dirichlet space $\mathcal{D}$ that are invariant under the forward shift operator \cite{RichterSundberg1991}.

These two classes of function spaces offer complementary perspectives. The de Branges--Rovnyak spaces provide explicit formulas for their reproducing kernels, but their norms are generally difficult to compute. In contrast, the harmonically weighted Dirichlet spaces often lack explicit kernel representations, yet their norms are described by concrete integral formulas. In 1997, Sarason discovered that the local Dirichlet space $\mathcal{D}_\zeta$ is, in fact, a de Branges--Rovnyak space \cite{Sarason1997LocalDS}. This finding sparked further exploration into the connections between these two classes of spaces. Subsequently, in 2010, Chevrot, Guillot, and Ransford demonstrated that Sarason’s example is the only case where the equality $\mathcal{H}(b) = \mathcal{D}_\mu$ holds with equality of norms \cite{Chevrot2010}. However, in 2013, Costara and Ransford showed that these spaces may still coincide as sets even when their norms are merely equivalent \cite{costara2013}. In particular, they established various sufficient and necessary conditions for the identity $\mathcal{H}(b) = \mathcal{D}_\mu$, with special attention to the case where $b$ is a rational function and $\mu$ is a {\em finitely} supported measure.

Subsequent research has examined several related questions. The case where $\mu$ is finitely supported was further studied in \cite{Eugenionext} and \cite{Lanucha2017}. Connections between higher-order local Dirichlet spaces and de Branges--Rovnyak spaces have been explored in \cite{hartmann2025}, \cite{LGR} and \cite{anucha2023DeBS}. The relationship between de Branges--Rovnyak spaces and subharmonically weighted Dirichlet spaces was analyzed in \cite{ElFallah2015DirichletSW} and \cite{Pouliasis2024WeightedDS}. However, no significant advances in the identification of de Branges--Rovnyak spaces with harmonically weighted Dirichlet spaces have emerged since 2013. In this article, we study the de Branges--Rovnyak spaces $\mathcal{H}(b)$ in the special case where $b = (1+u)/2$ and $u$ is an inner function. This is a natural framework to consider infinitely supported measures $\mu$, since then the Pythagorean mate $a=\gamma (1-u)/2$ (see \eqref{E:def-de-a}) vanishes at all the Clark points, and these points constitute an infinite set in our situation. We provide a complete characterization of those one-component inner functions $u$ for which the space $\mathcal{H}(b)$ coincides with a harmonically weighted Dirichlet space $\mathcal{D}_\mu$, for some appropriate measure $\mu$. A striking and novel feature of our result is that the measure $\mu$ may have {\em infinitely many} atoms, a phenomenon observed here for the first time for this equality of spaces. 

A central ingredient in our arguments is Bessonov’s characterization of one component inner functions based on the associated Clark measure \cite{bessonov2015duality}. Besides some natural conditions, his description involves a $T(1)$-condition, and another feature of our work is to deduce the boundedness of the Cauchy transform on $L^2(\sigma)$ when $\sigma$ is the Clark measure of a one-component inner function. In view of Bessonov’s result this can be reinterpreted as a $T(1)$-type result.

We conclude this section by clarifying the notations used throughout the paper. For a measurable set $M \subset \mathbb{T}=\partial\mathbb{D}$, we denote by $\chi_M$ its characteristic function and by $|M|$ its Lebesgue measure. In particular, for an arc $I \subset \mathbb{T}$, $|I|$ refers to its arc-length. When integrating with respect to Lebesgue measure, we write $\operatorname{d}\!m$. The norm of a function $f$ in a Banach space $X$ is denoted by $\|f\|_X$. For an operator $T$ on $X$, $\|T\|$ denotes its operator norm. Finally, the notation $f \lesssim g$ (or $g \gtrsim f$) means there exists a constant $C > 0$, independent of $f$ and $g$, such that $f \leq Cg$. If both $f \lesssim g$ and $f \gtrsim g$ hold, we write $f \asymp g$.

\section{Main results} \label{S:Main-resulats}
We now present the main results of this article. To this end, we begin with some preliminary definitions. However, rigorous definitions and additional properties of the objects involved will be given in the subsequent sections. Further novel results appear throughout the paper, either as propositions or as theorems. In this section, we restrict ourselves to highlighting the principal achievements.

Let $u$ be an inner function, i.e., a bounded analytic function in $H^\infty$ that is unimodular almost everywhere on the unit circle $\mathbb{T}$. Its Clark measure $\sigma^\alpha$, associated with the parameter $\alpha \in [0,1)$, is the unique non-negative finite singular Borel measure on $\mathbb{T}$ satisfying
\[
\frac{1 - |u(z)|^2}{|e^{2\pi i \alpha} - u(z)|^2} := \int_{\mathbb{T}} \frac{1 - |z|^2}{|\xi - z|^2} \, \operatorname{d}\!\sigma^\alpha(\xi), \qquad z \in \mathbb{D}.
\]
For simplicity, we write $\sigma = \sigma^0$. If $\sigma^\alpha$ has the form
\[
\sigma^\alpha = \sum_{n} \sigma_{\zeta_n}^\alpha \, \delta_{\zeta_n},
\]
where $\zeta_n \in \mathbb{T}$ and $\sigma_{\zeta_n}^\alpha > 0$, we say that $\sigma^\alpha$ is {\em discrete}. Moreover, in this case, $u$ admits a non-tangential angular derivative in the sense of Carath\'{e}odory at each $\zeta_n$, with boundary value $u(\zeta_n) = e^{2\pi i \alpha}$. The mass at the point $\zeta_n$, which we call an \emph{atom} of $\sigma$, is given by
\[
\sigma_{\zeta_n}^\alpha = \frac{1}{|u'(\zeta_n)|}.
\]
When there is no ambiguity, we simply write $\sigma_n^\alpha = \sigma_{\zeta_n}^\alpha$. For comprehensive introductions to Clark measures, we refer to \cite[Chapter 9]{cima2006cauchy}, \cite{Ross2013LENSLO}, and \cite{saksman2007}. Given a measure $\mu$ on $\mathbb{T}$, we define the associated potential by
\begin{equation} \label{E:potentialdef}
V_\mu(z) := \int_{\mathbb{T}} \frac{1}{|z - \xi|^2} \, \operatorname{d}\!\mu(\xi), \qquad z \in \mathbb{C}.
\end{equation}
Thus, we may consider the potentials associated with Clark measures.

We focus on inner functions $u$ that are \emph{one-component}, meaning that there exists an $\epsilon \in (0,1)$ such that the sublevel set
\[
\Omega(u, \epsilon) := \left\lbrace z \in \mathbb{D} \colon |u(z)| < \epsilon \right\rbrace
\]
is connected. The notion of one-component inner functions was introduced by Cohn \cite{cohn} in relation to Carleson embeddings of model spaces. These functions have since been studied extensively, see, for example, \cite{Aleksandrov2000, BARANOV2005116, Bellavita2022O, bessonov2015duality, cimamortini, nicolau2021}. In this context, the function
\[
b(z) := \frac{1 + u(z)}{2}, \qquad z \in \mathbb{D},
\]
is a non-extreme point of the closed unit ball of $H^\infty$, which is equivalent to $\int_\mathbb{T}\log(1-|b|)>-\infty$. For such a function $b$, there exists a so-called \emph{Pythagorean mate} $a$, i.e., a unique outer function satisfying $a(0) > 0$ and
\[
|b|^2 + |a|^2 = 1 \qquad \text{a.e. on } \mathbb{T}.
\]
In this case, it follows immediately that
\begin{equation} \label{E:def-de-a}
a(z) = \gamma \frac{1 - u(z)}{2}, \qquad z \in \mathbb{D},
\end{equation}
where $\gamma \in \mathbb{T}$ is a constant such that $(1 - u(0)) \gamma > 0$.

Our first main result is the following.

\begin{theorem} \label{T:main1}
Let $u$ be a one-component inner function and let $b=(1+u)/2$ with its Pythagorean mate $a$ given by \eqref{E:def-de-a}. Let $\{\zeta_n\}_n \subset \mathbb{T}$ be the points where $u(\zeta_n)=1$. Then
\[
\mathcal{H}(b)=\mathcal{D}_\mu
\]
if and only if $\mu$ is discrete with explicit expression
\begin{equation}\label{E:main1expressionmu}
 \mu = \sum_{n}\mu_n\delta_{\zeta_n},
\end{equation}
where its masses satisfy
\begin{equation} \label{E:main1atoms}
\frac{1}{C|u'(\zeta_n)|^2}\leq \mu_n \leq \frac{C}{|u'(\zeta_n)|^2}
\end{equation}
for a positive constant $C$, and its potential fulfills
\begin{equation}\label{E:main1potential}
\sup_{z \in \mathbb{D}} |a(z)|^2\, V_\mu(z) <\infty.
\end{equation}
\end{theorem}

In the case where $u$ is a finite Blaschke product, the function $b$ is rational, and Theorem \ref{T:main1} reduces to Theorem 4.1 of \cite{costara2013}; see also \cite{Eugenionext} for further details. We note that the upper estimate in \eqref{E:main1atoms} is actually redundant, as it follows directly from \eqref{E:main1potential}. Moreover, Lemma \ref{L:inf1c} implies that under the condition \eqref{E:main1atoms}, a lower bound holds in \eqref{E:main1potential}: specifically, $\inf_{\mathbb{D}}|a|^2 V_\mu > 0$. Nonetheless, we choose to present the theorem in this form to emphasize the constraints imposed on the weight coefficients. A key feature of Theorem \ref{T:main1} is that the support of the measure $\mu$ coincides with the support of the Clark measure $\sigma$, although the respective masses may differ within a specified range, essentially behaving like the squares of the Clark masses. A particularly important instance where Theorem \ref{T:main1} applies is the case of a singular inner function supported at a single point, that will be discussed in Subsection \ref{Subsection-example}. 

Additionally, we establish the following perturbation result for one-component inner functions, in the spirit of stability results on orthonormal or Riesz bases for model spaces (see \cite{Baranovstability}). This result appears to be new and of independent interest. The existence of the constants $A_\sigma$ and $B_\sigma$ appearing in the statement below is ensured by Bessonov’s characterization of one-component inner functions (see Theorem \ref{T:Bessonov}).

\begin{theorem}\label{T:perturbation}
Let $u$ be a one-component inner function, let $\{\zeta_n\}_n$ be its Clark atoms, $\sigma$ its Clark measure, and $A_\sigma,B_\sigma$ be constants for which
\[
A_\sigma \max\left(|\zeta_n-\zeta_n^{+}|, |\zeta_n-\zeta_n^{-}|\right) \leq \sigma_n\leq B_\sigma \min\left(|\zeta_n-\zeta_n^{+}|, |\zeta_n-\zeta_n^{-}|\right).
\]
Fix a sequence of positive numbers $\alpha=(\alpha_n)_n\in \ell^\infty$ such that
\begin{enumerate}[(i)]
\item \label{E:pertubation1} $\|\alpha\|_{\infty} \leq\min\{(3B_\sigma)^{-1},A_\sigma/3B_\sigma^2, 1/2\}$,
\item \label{E:pertubation2} and
\[
\sup_{n} \sum_{m\neq n}\frac{\sigma_m\alpha_m}{|\zeta_n-\zeta_m|}<\infty.
\]
\end{enumerate}
Then for every sequence of points $t_n\in\mathbb{T}$ such that $|t_n-\zeta_n|\leq \sigma_n \alpha_n$ and positive numbers $\lambda_n=\sigma_n + \epsilon_n$, with $|\epsilon_n|\leq\sigma_n \alpha_n$, the measure
\[
\lambda=\sum_{n}\lambda_n \delta_{t_n}
\]
is the Clark measure associated to a one-component inner function $\theta$.
\end{theorem}

As a consequence of the previous theorem and the example in Subsection \ref{Subsection-example}, we construct a whole family of examples for which the identity $\mathcal{H}(b) = \mathcal{D}_\mu$ holds and, more importantly, the measure $\mu$ possesses infinitely many atoms.

In the proof of Theorem \ref{T:main1}, a central role is played by the (truncated) Cauchy transform defined by
\begin{equation}\label{E:Cauchydef}
\mathcal{C}_\sigma f(\zeta) := \int_{\mathbb{T}\setminus\{\zeta\}}\frac{f(s)}{1-\overline{s}\zeta}\operatorname{d}\!\sigma(s),\qquad \zeta\in\mathbb{T},
\end{equation}
which acts from $L^2(\sigma)$ to itself. For necessary and sufficient conditions regarding the boundedness of $\mathcal{C}_\sigma$, as well as related problems, we refer the reader to \cite{LaceyII, Lacey2014, Ryan2025, Tolsa2001} and the references therein. In our setting, leveraging a result of Tolsa from \cite{Tolsa2001}, we establish that $\mathcal{C}_\sigma$ is indeed bounded when $\sigma$ is the Clark measure associated with a one-component inner function.

\begin{theorem}\label{T:main2}
Let $u$ be a one-component inner function, and let $\sigma$ be its Clark measure. Then the Cauchy transform
\[\mathcal{C}_\sigma f(\zeta)= \int_{\mathbb{T}\setminus\{\zeta\}}\frac{f(s)}{1-\overline{s}\zeta}\operatorname{d}\!\sigma(s), \qquad\zeta\in\mathbb{T},\]
is bounded on $L^2(\sigma)$.
\end{theorem}

This result was initially motivated by Theorem 1 of \cite{bessonov2015duality} (see also Theorem \ref{T:Bessonov} below), where the fifth condition takes the form of a $T(1)$ criterion. To the best of our knowledge, Theorem \ref{T:main2} is new and of independent interest, as it highlights an additional remarkable property of one-component inner functions. Its conclusion does not follow directly from the boundedness of $\mathcal{C}_\sigma 1$ or from previously known results in the extensive literature on the subject.

The organization of the rest of the paper is as follows. Section \ref{S:preliminaries} provides background on de Branges--Rovnyak spaces, harmonically weighted Dirichlet spaces, and one-component inner functions. While most of these results are already established, we include proofs for those lacking readily accessible references. The proof of Theorem \ref{T:main2} is given in Section \ref{S:proofT2}. Sections \ref{S:proofT1measure} through \ref{S:proofT1embDm} are devoted to the proof of Theorem \ref{T:main1}. Specifically, in Section \ref{S:proofT1measure}, we demonstrate the necessity of conditions \eqref{E:main1expressionmu}, \eqref{E:main1atoms}, and \eqref{E:main1potential}. Sections \ref{S:proofT1embHb} and \ref{S:proofT1embDm} then provide necessary and sufficient conditions for the embeddings $\mathcal{H}(b)\hookrightarrow \mathcal{D}_\mu$ and $\mathcal{D}_\mu\hookrightarrow \mathcal{H}(b)$, respectively, which together allow us to complete the proof of Theorem \ref{T:main1} in Section \ref{S:proofT1embDm}. In Section \ref{S:examples}, we present examples illustrating Theorem \ref{T:main1}. We construct an explicit function satisfying its conditions, thereby showing that the associated de Branges--Rovnyak space is also a harmonically weighted Dirichlet space. In view of Theorem \ref{T:main2}, one might suspect that, since the Cauchy transform is automatically bounded in the setting of one-component inner functions, the potential condition \eqref{E:main1potential} is likewise automatically satisfied. However, in Section \ref{S:examples}, we also construct a class of one-component inner functions for which \eqref{E:main1potential} fails. These examples demonstrate the optimality of the condition stated in Theorem \ref{T:main1}. In Section \ref{S:perp-result}, we prove Theorem \ref{T:perturbation} and present a broad class of examples in which the equality $\mathcal{H}(b) = \mathcal{D}_\mu$ holds and the measure $\mu$ possesses infinitely many atoms. In Section \ref{S:brownshields}, as an application of Theorem \ref{T:main1}, we verify a variation of the Brown–Shields conjecture for these de Branges--Rovnyak spaces. We emphasize that our result, Theorem \ref{T:main3}, is a direct consequence of Theorem 1 in \cite{ELFALLAH20163262}.

\section{Preliminaries} \label{S:preliminaries}

\subsection{De Branges--Rovnyak spaces}
Let $b$ be an $H^\infty$-function on the open unit disk $\mathbb{D}$ with $\|b\|_{H^\infty}\leq1$. The de Branges--Rovnyak space $\mathcal{H}(b)$ is the reproducing kernel Hilbert space on $\mathbb{D}$ associated with the positive definite kernel
\[
k^b_\lambda(z) := \frac{1-b(z)\overline{b(\lambda)}}{1-z\overline{\lambda}},
\]
where $z, \lambda \in \mathbb{D}$.

Though $\mathcal{H}(b)$ is contractively contained in the classical Hardy space $H^2$, it is generally not closed in the $H^2$ norm.
It is well-known that $\mathcal{H}(b)$ is closed in $H^2$ if and only if $b=u$ is an inner function.
In this case, $\mathcal{H}(b)=K_u=H^2\ominus uH^2$ (the orthogonal complement of $uH^2$ in $H^2$) is the so-called model space \cite{MR2500010}.

The model spaces can be equivalently introduced as the closed invariant subspaces of the backward shift operator \cite[Chapter 5]{model}, or using the Aleksandrov-Clark transform \cite[Chapter 11]{model}. Indeed, if $\sigma^\alpha$ is a Clark measure associated to $u$, then $K_u=W_{\sigma^\alpha}(L^2(\sigma^\alpha))$, where, for every $f \in L^2(\sigma^\alpha)$,
\[
W_{\sigma^\alpha} f(z) := (1-e^{-2\pi i\alpha}u(z))\int_{\mathbb{T}}\frac{f(\xi)}{1-z\overline{\xi}}\,\operatorname{d}\!\sigma^\alpha(\xi),
\]
or, equivalently,
\[
W_{\sigma^\alpha} f(z)=\int_{\mathbb{T}}\frac{1-u(z)\overline{u(\xi)}}{1-z\overline{\xi}}f(\xi)\,\operatorname{d}\!\sigma^\alpha(\xi).
\]
Moreover, for every $f \in K_u$,
\[
W_{\sigma^\alpha}^{-1}f=f_{\vert_{\operatorname{supp}(\sigma^\alpha)}}, \qquad \sigma^\alpha\text{-a.e. on}\,\,\mathbb{T},
\]
and
\begin{equation}\label{Clark measure}
\|f\|_{K_u} = \|W_{\sigma^\alpha}^{-1}f\|_{L^2(\sigma^\alpha)} = \|f\|_{L^2(\sigma^\alpha)}.
\end{equation}
See \cite{Poltoratski} for precise definitions of $W^{-1}_{\sigma^\alpha}$.

Given an inner function $u$ and $b=(1+u)/2$, the identity
\[
\left|\frac{1+u(z)}{2}\right|^2+\left|\frac{1-u(z)}{2}\right|^2=\frac{1+|u(z)|^2}{2},\qquad z\in\mathbb{D},
\]
shows that the Pythagorean mate of $b$ is
\[
a(z)=\gamma\frac{1-u(z)}{2},
\]
with $\gamma\in\mathbb{T}$ chosen so that $a(0)>0$. In this case $(b,a)$ clearly forms a Corona pair, i.e.,
\[
\inf_{z \in \mathbb{D}}\left(|a(z)|^2+|b(z)|^2\right)>0.
\]
With this choice of $b$, the following result in \cite{FricainGrivaux} explicitly describes the space $\mathcal{H}(b)$. In fact, this result was initiated in \cite{Fricain_Hartmann_Ross_2018}.

\begin{lemma}[Fricain--Grivaux, Lemma 2.2 of \cite{FricainGrivaux}]\label{L:FricainGrivaux}
Let $b=(1+u)/2$ with $u$ an inner function. Then $\mathcal{H}(b)$ decomposes as the orthogonal sum
\[
\mathcal{H}(b) = K_u \oplus_b (1-u)H^2.
\]
Moreover, if $f=g+(1-u)h$, with $g\in K_u$ and $h\in H^2$, we have
\begin{equation} \label{E:normeqHb}
\|f\|_{\mathcal{H}(b)}^2\asymp \|g\|_{H^2}^2+\|h\|_{H^2}^2.
\end{equation}
\end{lemma}

For more information about the de Branges--Rovnyak spaces, we refer to the monographs \cite{hb2, sarHb}.

\subsection{Harmonically weighted Dirichlet spaces}
Given a finite positive Borel measure $\mu$ on the unit circle $\mathbb{T}$, the associated harmonically weighted Dirichlet space $\mathcal{D}_\mu$ is the
family of all functions in $H^2$ that have finite harmonically weighted Dirichlet
integral, that is,
\[
\mathcal{D}_\mu(f) := \frac{1}{\pi}\int_{\mathbb{D}}|f'(z)|^2 P_\mu(z) \operatorname{d}\!A(z) <\infty,
\]
where 
\[
P_\mu(z)=\int_{\mathbb{T}}\frac{1-|z|^2}{|z-\zeta|^2}\operatorname{d}\!\mu(\zeta), \qquad z\in\mathbb{D},
\]
is the Poisson integral of $\mu$ and $\operatorname{d}\!A$ is the two-dimensional Lebesgue measure.
We notice that the weighted Dirichlet integral annihilates all the constants. Thus, on its own, it does not produce a norm. However, $\mathcal{D}_\mu$ is a Hilbert space with respect to the norm
\[
\|f\|_{\mathcal{D}_\mu}^2 := \|f\|^2_{H^2}+\mathcal{D}_\mu(f).
\]

Choosing as $\mu$ the Lebesgue measure on $\mathbb{T}$, one sees that $\mathcal{D}_\mu$ coincides with the classical Dirichlet space. For $\zeta \in \mathbb{T}$, considering the Dirac mass $\delta_\zeta$, we obtain the so-called local Dirichlet space, which we simply denote by $\mathcal{D}_\zeta$. It is a known fact that polynomials are dense in $\mathcal{D}_\mu$, for every choice of $\mu$. For further information about the harmonically weighted Dirichlet spaces, we refer to \cite[Chapter 7]{primer-fkmr}.

To continue our analysis, we restrict our attention to a special class of measures, more precisely, we assume that $\mu$ has the form
\begin{equation*}
\mu=\sum_{j\in\mathbb{N}} \mu_j \delta_{\zeta_j},
\end{equation*}
where $\mu_j>0$, $\sum_{j\in\mathbb{N}} \mu_j <\infty$, and $\zeta_j\in\mathbb{T}$. With this assumption, the potential $V_\mu$ has the explicit form
\begin{equation*}
V_\mu(z)=\sum_{j\in\mathbb{N}} \frac{\mu_j}{|z-\zeta_j|^2}, \qquad z\in\mathbb{C}.
\end{equation*}
Note that $V_\mu(\zeta_n)=\infty$ for every point $\zeta_n$ and $V_\mu(z)<\infty$ for every $z\notin\overline{\cup_j\zeta_j}$. Also, $V_\mu$ is lower semi-continuous on $\mathbb{C}$ and it is continuous on $\mathbb{C}\setminus\overline{\cup_j\zeta_j}$. Finally, observe that when $\nu=\sum_{j\in\mathbb{N}} \nu_j \delta_{\zeta_j}$ with $\nu_j\asymp\mu_j$, then $\mathcal{D}_\mu=\mathcal{D}_\nu$.

\subsection{One-component inner functions}
In Theorems \ref{T:main1} and \ref{T:main2}, we considered one-component inner functions. The associated model spaces and Clark measures in this setting possess several important properties, which we collect in this section. In the whole article, when $\zeta\in \mathbb{T}$, by $u'(\zeta)$ we mean the angular derivative of $u$ in the sense of Carath\'{e}odory at the boundary point $\zeta$. We also write $|u'(\zeta)|=\infty$ if such a derivative does not exist. By Julia's inequality, we always have $|u'(\zeta)|>0$. The norm of the reproducing kernel $k^u_z$ is equal to
\begin{equation}\label{E:norm-ku}
\|k^u_z\|^2_{H^2}=\frac{1-|u(z)|^2}{1-|z|^2}, \qquad z \in \mathbb{D}.
\end{equation}
If $\zeta \in \mathbb{T}$ with $|u'(\zeta)|<\infty$, then $k^u_\zeta$ is well defined and belongs to $K_u$ with
\[
\|k^u_\zeta\|^2_{H^2}=|u'(\zeta)|.
\]
See \cite[Theorem 21.1]{hb2} and \cite{MR2986324}.

If $\|b\|_{H^\infty}<1$, then  $\mathcal{H}(b)$ coincides with $H^2$ as sets with an equivalent norm, and thus this case is not interesting for the purposes of this work. Hence, from now on we consider only symbols with $\|b\|_{H^\infty}=1$. In this setting, the \emph{boundary spectrum} is
\[
\rho(b) := \{\zeta \in \mathbb{T}\ \colon\ \liminf_{z \to \zeta}|b(z)|<1\}.
\]
When $u$ is an inner function, $\rho(u)$ is a closed subset of $\mathbb{T}$ and can be equivalently described as
\[
\rho(u)=\{\zeta \in \mathbb{T}\ \colon\ \liminf_{z \to \zeta}|u(z)|=0\}.
\]
It is well known that every function $f\in K_u$ admits an analytic extension across any open arc $I \subset \mathbb{T}\setminus \rho(u)$. If $u$ is one-component, then by Theorem 1.11 of \cite{Aleksandrov2000}, $|\rho(u)|=0$ and $\liminf_{r\to 1^-}|u(r\zeta)|<1$ for all $\zeta\in\rho(u)$. In particular, this implies $|u'(\zeta)|=\infty$ for every $\zeta\in\rho(u)$. According to \cite[Theorem 1.2]{Aleksandrov2000}, there also exists a constant $C_u>0$ such that
\begin{equation}\label{E:Relation norms kernels one-component}
\|k^u_z\|_{H^\infty}\leq C_u\|k^u_z\|_{H^2}^2
\end{equation}
for every $z \in \overline{\mathbb{D}}\setminus\rho(u)$.  We may choose $C_u$ to be the smallest number that satisfies the above inequality, and this quantity will appear in several discussions in the rest of the paper. When there is no ambiguity, we will also write $C$ for $C_u$. Moreover, by \eqref{E:norm-ku} and \eqref{E:Relation norms kernels one-component},
\begin{equation} \label{E:aleksandrov3}
\frac{1-|z|^2}{1-|u(z)|^2} \leq C_u \, \inf_{w\in\mathbb{D}}\left|\frac{1-\overline{z}w}{1-\overline{u(z)}u(w)}\right|, \qquad z\in\mathbb{D}.
\end{equation}
To exploit this estimation, fix $\zeta\in\mathbb{T}\setminus\rho(u)$, and let $\eta\in\rho(u)$ be such that $\operatorname{dist}(\zeta,\rho(u))=|\zeta-\eta|$. Then
\[
\inf_{w\in\mathbb{D}}\left|\frac{1-\overline{z}w}{1-\overline{u(z)}u(w)}\right|\leq \liminf_{w\to\eta} \left|\frac{1-\overline{z}w}{1-\overline{u(z)}u(w)}\right|=|z-\eta|, \qquad z\in\mathbb{D}.
\]
Hence, taking the limit in \eqref{E:aleksandrov3} as $z\to\zeta\in\mathbb{T}$, we obtain
\[
\frac{1}{|u'(\zeta)|}\leq C_u \text{ dist}(\zeta,\rho(u)) ,\qquad \zeta \in \mathbb{T}\setminus\rho(u).
\]
An analogous result also follows from the recent paper \cite{bellavita2024spectralanalysisdifferencequotient}. Interpreting $1/\infty = 0$ in the above inequality corresponding to the case $\zeta\in\rho(u)$, we may even write
\begin{equation}\label{dist spectrum}
\frac{1}{|u'(\zeta)|}\leq C_u \text{ dist}(\zeta,\rho(u)) ,\qquad \zeta \in \mathbb{T}.
\end{equation}

The Clark measures associated to one-component inner functions have been completely characterized by Bessonov. For a Borel measure $\mu$ on $\mathbb{T}$, we denote its support by $\operatorname{supp}(\mu)$, and
\[
a(\mu)=\{\zeta\in\mathbb{T}\colon \zeta \,\, \text{is an isolated atom of}\,\,\mu\},
\]
and $\tau(\mu)=\operatorname{supp}(\mu)\setminus a(\mu)$. For all Clark measures $\sigma^\alpha$ of the one-component inner function $u$, regardless of the parameter $\alpha$, we have $\tau(\sigma^\alpha)=\rho(u)$ \cite[Lemma 2.1]{bessonov2015duality}. For the complementary part $a(\sigma^\alpha)$, we have the following complete description.

\begin{theorem}[Bessonov, Theorem 1 in \cite{bessonov2015duality}]\label{T:Bessonov}
Let $\alpha \in [0,1)$. The following conditions are necessary and sufficient for a Borel measure $\mu$ on $\mathbb{T}$ to be the Clark measure $\sigma^\alpha$ of a one-component inner function.
\begin{enumerate}[(i)]
\item $|\!\operatorname{supp} (\mu)|= 0$.
\item $\mu$ is a discrete measure with only isolated atoms.
\item Every atom $\xi$ has two neighbors\footnote{ By ‘‘neighbors", we mean that there are no other atoms between both $\xi$ and $\xi^+$, and between $\xi$ and~$\xi^-$.} $\xi^\pm \in a(\mu)$, and every connected component of $\mathbb{T}\setminus \tau(\mu)$ contains atoms of $\mu$.
\item \label{T:Bessonov(iv)}There exist positive constants $A_\mu$ and $B_\mu$ such that for every $\xi \in a(\mu)$
\[
A_\mu|\xi-\xi^\pm|\leq \mu(\xi)\leq B_\mu|\xi-\xi^\pm|.
\]
\item The Cauchy transform
\[
\mathcal{C}_\mu 1(z)= \int_{\mathbb{T}\setminus\{z\}}\frac{\operatorname{d}\!\mu(s)}{1-\overline{s}z}, \qquad z\in\mathbb{T},
\]
is uniformly bounded on $a(\mu)$.
\end{enumerate}
\end{theorem}

Fix $N\in\mathbb{N}$ and let $\ell$ be a natural number in $\{0, \ldots, N-1\}$. We define
\begin{equation}\label{T_ell}
T_{\ell,N}=\{t \in \mathbb{T}\setminus \rho(u)\ \colon\ u(t)=e^{2\pi i\frac{\ell}{N}}\}
\end{equation}
and
\begin{equation}\label{t_m}
T_N=\bigcup_{\ell=0}^{N-1}T_{\ell,N}=\{t_m\}_m.
\end{equation}
Note that $T_{\ell,N}$ consists of all the base points for the atoms of the Clark measure $\sigma^{\ell/N}$. We partition $\mathbb{T}\setminus \rho(u)$ into arcs $J_n$ with mutually disjoint interiors and whose endpoints are in the set $T_N$. We denote $J_n=[t_{n_1},t_{n_2})$ the arc having endpoints $t_{n_1}$ and $t_{n_2}$, the latter being excluded, so that
\[
u(t_{n_1})=e^{2\pi i\frac{k_n}{N}}, \quad u(t_{n_2})=e^{2\pi i\frac{k_n+1}{N}},
\]
for an appropriate $k_n\in \{0,\ldots,N-1\}$, and thus
\[
\frac{1}{2\pi}\int_{J_n}|u'(\xi)|\,|\!\operatorname{d}\!\xi|=\frac{1}{N}.
\]
Note that, under these conditions, the points $t_{n_1},t_{n_2}$ are consecutive, i.e., $(t_{n_1},t_{n_2})\cap T_N=\emptyset.$

\begin{lemma}[Baranov--Dyakonov, Lemma 5.1 of \cite{baranov2011feichtinger}]\label{L:feichtinger}
Let $u$ be a one-component inner function. Let $N\geq 20 \pi C$, with $C$ as in \eqref{dist spectrum}. Then, for every $t \in J_n=[t_{n_1},t_{n_2})$,
\[\frac{1}{C_1}|u'(t_{n_1})|\leq  |u'(t)|\leq C_1 |u'(t_{n_1})|,\]
and
 \[\frac{1}{C_2}\frac{1}{N|u'(t_{n_1})|}\leq |J_n|\leq C_2\frac{1}{N|u'(t_{n_1})|}.\]
 We may take $C_1=100/81$ and $C_2=2\pi C_1$.
\end{lemma}

With the help of Lemma \ref{L:feichtinger}, we are able to estimate $\sigma(Q)$ for any arc $Q \subset \mathbb{T}$.

\begin{lemma}\label{Lemma misura}
Let $u$ be a one-component inner function, let $\alpha \in (0,1)$, and let $Q\subset \mathbb{T}$ be an arc which contains at least one atom of the Clark measure $\sigma$ and one atom of the Clark measure $\sigma^\alpha$. Choose the integer $N\geq 20\pi C$, where $C$ is as in \eqref{dist spectrum}, and put $K=2C_1^{N}$, where $C_1$ is the constant in Lemma \ref{L:feichtinger}. Then
\[\frac{1}{K}\sigma^\alpha(Q)\leq \sigma(Q)\leq K\sigma^\alpha(Q).\]
Moreover, if $Q$ contains at least two atoms of $\sigma$,
\begin{equation} \label{E:ClarkLebesgue}
\frac{1}{K}\sigma(Q)\leq |Q|\leq K\sigma(Q).
\end{equation}
\end{lemma}

\begin{proof}[Proof of Lemma \ref{Lemma misura}]
The sequences of atoms associated to $\sigma$ and $\sigma^\alpha$ are intertwined, that is, for a suitable numbering,
\[
[\xi_i,\xi_i^+]\cap \operatorname{supp}(\sigma^\alpha)=\xi_i^\alpha \ , \ [\xi_i^\alpha,{\xi_i^\alpha}^+]\cap \operatorname{supp}(\sigma)=\xi_{i}^+.
\]
Let $T_N=\{t_m\}_m$ be defined as in \eqref{t_m}. For the sake of clarity, below we denote $t_n,t_{n+1}$ two consecutive points in $T_N$.

\begin{figure}
\centering
\includegraphics[width=1\linewidth]{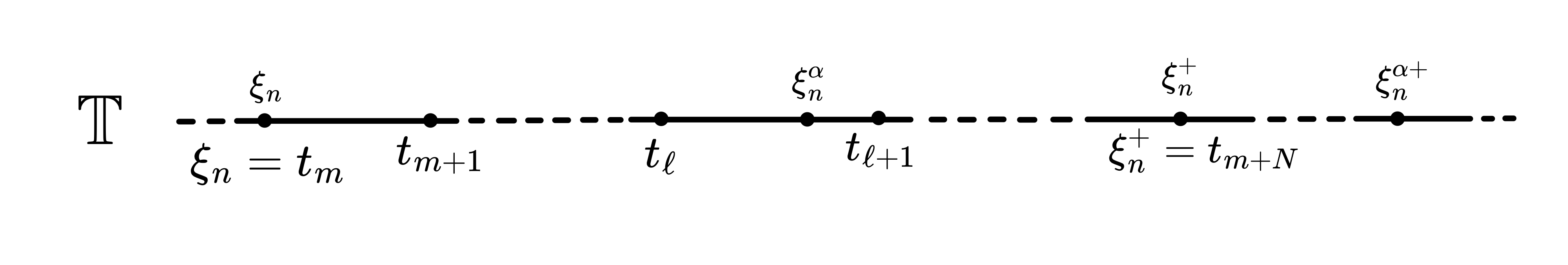}
\caption{Respective locations of $\xi_n$, $\xi_n^+$ and $\xi_n^\alpha,\xi_n^{\alpha+}$.}
\label{Figure1}
\end{figure}

We consider an atom $\xi_n\in Q$. Then $\xi_n=t_{m},\xi_n^+=t_{m+N}$ for some $m$, and the intervals $J_k=[t_{k},t_{k+1})$, $k=m,\ldots,m+N-1$ form a partition of $[\xi_n,\xi_{n}^+)$. In particular, an atom $\xi_n^\alpha$ of the measure $\sigma^\alpha$ must be in one of the intervals $[t_\ell,t_{\ell+1})$, for some $\ell\in\{m,m+1,\ldots,m+N-1\}$. See Figure \ref{Figure1}. Thus, an iterated application of Lemma \ref{L:feichtinger} gives
\begin{align*}
\frac{1}{|u'(\xi_n)|}=\frac{1}{|u'(t_m)|}\leq \frac{C_1}{|u'(t_{m+1})|}\leq \dots \leq  \frac{C_1^{\ell-m}}{|u'(t_{\ell})|}\leq  \frac{C_1^{\ell-m+1}}{|u'(\xi_n^\alpha)|}\leq   \frac{C_1^{N}}{|u'(\xi_n^\alpha)|}.
\end{align*}
Arguing analogously, we obtain
\[
\frac{1}{|u'(\xi_{n}^+)|}\leq C_1^{N}\frac{1}{|u'(\xi^\alpha_{n})|}.
\]
Consequently, since for each $\xi_i \in Q$, there is at least one $\xi_i^\alpha \in Q$ either in $(\xi_i,\xi_i^+)$ or in $(\xi_i^-,\xi_i)$, we conclude that
\[
\sigma(Q)= \sum_{\xi_i \in Q} \sigma_{\xi_i}\leq \sum_{\xi^\alpha_i \in Q} 2C_1^{N} \sigma^\alpha_{\xi^\alpha_i}=2C_1^{N}\sigma^\alpha(Q).
\]
For the reverse implication, we argue analogously.

Finally, using the Aleksandrov-Clark disintegration formula \cite[Chapter 9.3]{cima2006cauchy}, we have
\[
|Q|=\int_{0}^1 \sigma^\alpha(Q)\operatorname{d}\!\alpha \leq 2C_1^{N}\sigma(Q)\int_{0}^1 \operatorname{d}\!\alpha=2C_1^{N}\sigma(Q)
\]
and
\[
|Q|=\int_{0}^1 \sigma^\alpha(Q)\operatorname{d}\!\alpha \geq \frac{1}{2C_1^{N}}\sigma(Q)\int_{0}^1\operatorname{d}\!\alpha=\frac{1}{2C_1^{N}}\sigma(Q).
\]
Notice that the assumption that $Q$ contains two atoms of $\sigma$ ensures that $Q$ contains at least one atom of $\sigma^\alpha$ for every $\alpha\in (0,1)$, a fact which was implicitly used in the above estimations.
\end{proof}

Naively speaking, Lemma \ref{Lemma misura} says that when the arc $Q$ is sufficiently large, the action of $\sigma$ on it is like the Lebesgue measure. In particular, the estimates in \eqref{E:ClarkLebesgue} show that $\sigma$ has the doubling property on arcs $Q$ that contain at least two atoms, that is,
\[
\sigma(2Q)\lesssim\sigma(Q),
\]
where $2Q$ is the arc having the same center as $Q$ but double the arc-length. A relation between the doubling property and one-component inner functions was also observed in \cite[Subsection 2.1]{Marzo2012}.

\begin{lemma}\label{L:integralestimate}
Let $u$ be a one-component inner function, and let $Q\subset \mathbb{T}$ be an arc. Then, for every atom $\xi_i\in Q$,
\[
\int_{\mathbb{T}\setminus Q}\frac{1}{|\xi-\xi_i|^2}\operatorname{d}\!\sigma(\xi)\lesssim \frac{1}{\operatorname{dist}(\xi_i, \mathbb{T}\setminus Q)},
\]
where the constant involved only depends on $u$ and not on either $\xi_i$ or $Q$.
\end{lemma}

\begin{proof}
If $\xi_i\in\partial Q,$ then $\operatorname{dist}(\xi_i, \mathbb{T}\setminus Q)=0$ and the inequality is trivial. Thus, we may assume that $\xi_i$ lies in the interior of $Q$, and then all the considered quantities are finite. The assertion is trivially true when $\mathbb{T}\setminus~\!\!Q$ does not contain any atom (in particular, when $Q=\mathbb{T}$). Considering a suitable rotation $\tilde{u}(z)=u(\xi_i z)$, we may assume that $\xi_i=1$. Also, let $e^{ia}$ and $e^{ib}$ be the endpoints of $Q$, with $0<b\leq a<2\pi$, so that
\[
\overline{Q}=\{e^{i\theta}\colon a\leq\theta\leq 2\pi\}\cup\{e^{i\theta}\colon 0\leq\theta\leq b\}=:[e^{ia},e^{ib}].
\]
See Figure \ref{Figure2}. We divide the Clark atoms of $u$ contained in $\mathbb{T}\setminus Q$ into two subsets:
\begin{align*}
A_{\operatorname{POS}}&=\{\xi_n\in a(\sigma)\setminus Q\colon \Im(\xi_n)\geq 0\}, \\
A_{\operatorname{NEG}}&=\{\xi_n\in a(\sigma)\setminus Q\colon\Im(\xi_n)<0 \},
\end{align*}
so that
\[
\int_{\mathbb{T}\setminus Q}\frac{1}{|\xi-1|^2}\operatorname{d}\!\sigma(\xi)=\sum_{\xi \in A_{\operatorname{POS}}} \frac{1}{|1-\xi|^2}\sigma_{\xi} +\sum_{\xi \in A_{\operatorname{NEG}}} \frac{1}{|1-\xi|^2}\sigma_{\xi}.
\]

If $-1\notin\rho(u)$, then there exist atoms $\xi_{\operatorname{POS}},\xi_{\operatorname{NEG}}$ in $A_{\operatorname{POS}}$ and $A_{\operatorname{NEG}}$, respectively, that are closest to $-1$ (see Figure 2). In this case, $(\xi_{\operatorname{POS}})^+=\xi_{\operatorname{NEG}}$. If however $-1\in\rho(u)$, then such atoms do not exist. We will show that the presence of these two special atoms is inessential, as they can always be disregarded.

\begin{figure}
\centering
\includegraphics[width=0.6\linewidth]{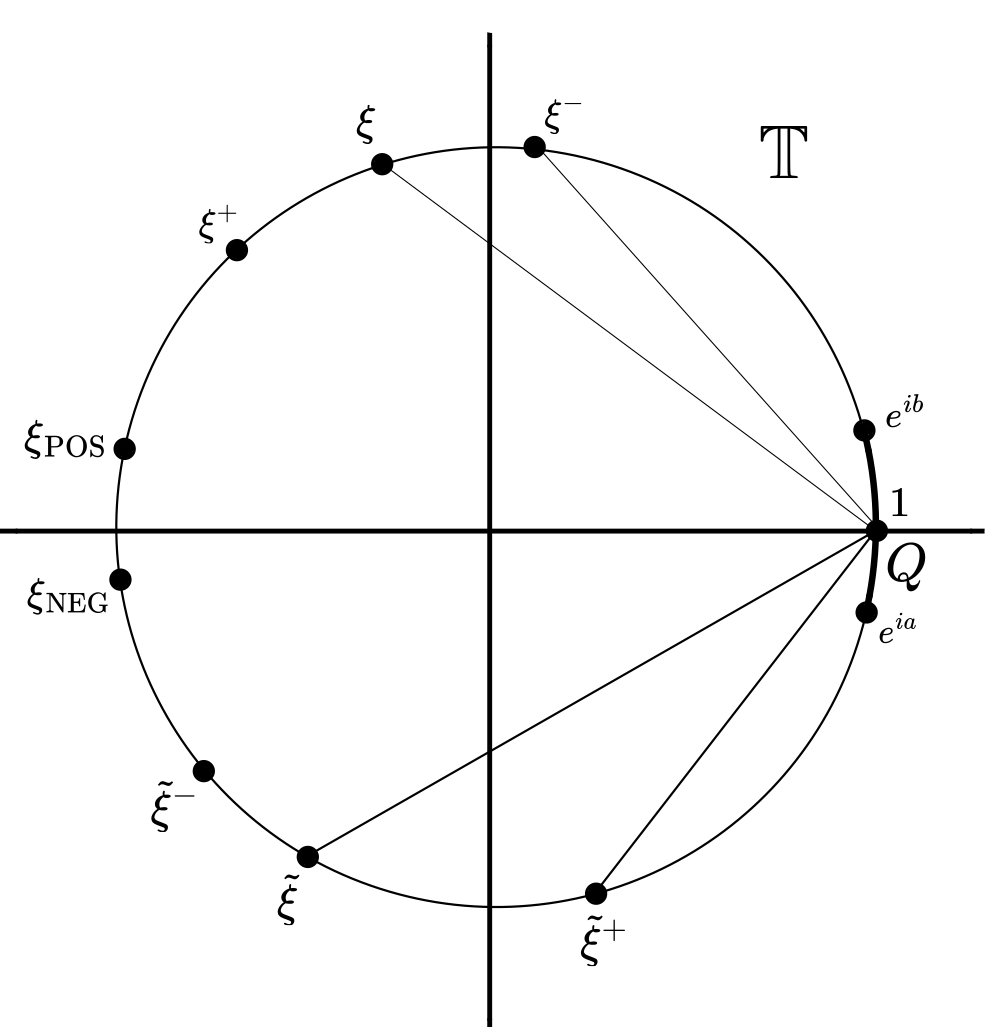}
\caption{Distribution of atoms when $-1\notin\rho(u)$.}
\label{Figure2}
\end{figure}

For every atom $\xi$ in $A_{\operatorname{POS}}$, we have $|1-\xi^-|\leq |1-\xi|$, and
\[
\frac{|1-\xi^+|}{|1-\xi|}\leq 1+ \frac{|\xi-\xi^+|}{|1-\xi|}\leq 1+\frac{B_\sigma}{A_\sigma}\frac{|\xi^- -\xi|}{|1-\xi|}\leq 1+\frac{B_\sigma}{A_\sigma},
\]
where in the second inequality we used the fourth condition of Theorem \ref{T:Bessonov}. Analogously, for every atom $\xi\in A_{\operatorname{NEG}}$, we have
$|1-\xi^+|\leq |1-\xi|$ and
\[
\frac{|1-\xi^-|}{|1-\xi|}\leq  1+\frac{B_\sigma}{A_\sigma}.
\]
We set $K=B_\sigma/A_\sigma$. In the computation below, if the ‘‘extremal” atoms $\xi_{\operatorname{POS}}$ or $\xi_{\operatorname{NEG}}$
do not exist, we treat the corresponding fraction containing
 $\sigma_{\xi_{\operatorname{POS}}}$ or $\sigma_{\xi_{\operatorname{NEG}}}$ as zero, and in any case the following upper estimates are valid. For the part in the upper half circle, we have
\begin{align*}
\sum_{\xi \in A_{\operatorname{POS}}} \frac{1}{|1-\xi|^2}\sigma_{\xi} &=\sum_{\xi \in A_{\operatorname{POS}}\setminus\{\xi_{\operatorname{POS}}\}} \frac{1}{|1-\xi|^2}\sigma_{\xi} + \frac{\sigma_{\xi_{\operatorname{POS}}}}{|1-\xi_{\operatorname{POS}}|^2}\\
&\leq\sum_{\xi \in A_{\operatorname{POS}}\setminus\{\xi_{\operatorname{POS}}\}} \frac{1}{|1-\xi|^2}\sigma_{\xi} + K\frac{\sigma_{(\xi_{\operatorname{POS}})^-}}{|1-(\xi_{\operatorname{POS}})^-|^2}\\
&\leq (1+K)\sum_{\xi \in A_{\operatorname{POS}}\setminus\{\xi_{\operatorname{POS}}\}} \frac{1}{|1-\xi|^2}\sigma_{\xi} \\
&\leq  B_\sigma(1+K) \sum_{\xi \in A_{\operatorname{POS}}\setminus\{\xi_{\operatorname{POS}}\}} \frac{(1+K)^2}{|1-\xi^+|^2}|\xi-\xi^+|.
\end{align*}
The first and third inequalities come again from Theorem \ref{T:Bessonov} \eqref{T:Bessonov(iv)} and the above estimates. Analogously, for $A_{\operatorname{NEG}}$, we have that
\[
\sum_{\xi \in A_{\operatorname{NEG}}} \frac{1}{|1-\xi|^2}\sigma_{\xi}\leq  B_\sigma \sum_{\xi \in A_{\operatorname{NEG}}\setminus\{\xi_{\operatorname{NEG}}\}} \frac{(1+K)^3}{|1-\xi^-|^2}|\xi-\xi^-|.
\]
Let $I_{\xi,\pm}$ be the arcs with extremes $\xi$ and $\xi^\pm$, respectively, and
\[
E_+=\bigcup_{\xi\in A_{\operatorname{POS}}\setminus\{\xi_{\operatorname{POS}}\}}I_{\xi,+},\qquad E_-=\bigcup_{\xi\in A_{\operatorname{NEG}}\setminus\{\xi_{\operatorname{NEG}}\}}I_{\xi,-}.
\]
For every $\xi\in A_{\operatorname{POS}}\setminus\{\xi_{\operatorname{POS}}\}$ and for every $\zeta=e^{it}\in I_{\xi,+}$, the trivial estimation $|1-\zeta|\leq|1-\xi^+|$ holds. Hence, introducing the step function
\[
f_{\operatorname{POS}}(\zeta)=\sum_{\xi\in A_{\operatorname{POS}}}\frac{1}{|1-\xi^+|^2}\chi_{I_{\xi,+}}(\zeta),
\]
we see that
\[
\sum_{\xi\in A_{\operatorname{POS}}\setminus\{\xi_{\operatorname{POS}}\}} \!\frac{1}{|1-\xi^+|^2}|\xi-\xi^+|=\int_{E_+}\!\!f_{\operatorname{POS}}(\zeta)\operatorname{d}\!m(\zeta) \leq\int_{E_+}\frac{1}{|1-\zeta|^2}\operatorname{d}\!m(\zeta),
\]
where now we are integrating with respect to the Lebesgue measure. After analogous considerations for the atoms in $A_{\operatorname{NEG}}$, we conclude that
\[
\int_{\mathbb{T}\setminus Q}\frac{1}{|\xi-1|^2}\operatorname{d}\!\sigma(\xi) \leq B_\sigma(1+K)^3\int_{E_+\cup E_-}\frac{1}{|1-\zeta|^2}\operatorname{d}\!m(\zeta).
\]
Now, by construction, the set $E_+\cup E_-$ is contained in $\overline{\mathbb{T}\setminus Q}$. The closure over the interval $\mathbb{T}\setminus Q$ appears because the endpoints of $Q$, \emph{a priori}, might belong to the spectrum $\rho(u)$, and hence there may exist a sequence of atoms in $A_{\operatorname{POS}}$ and/or $A_{\operatorname{NEG}}$ converging to an endpoint of~$Q$. Recall that the endpoints of $Q$ are $e^{ia},e^{ib}$ with $0<b\leq a<~2\pi$. A direct computation yields that

\begin{align*}
\int_{\mathbb{T}\setminus Q}\frac{1}{|\xi-1|^2}\operatorname{d}\!\sigma(\xi)&\leq B_\sigma(1+K)^3 \int_{\mathbb{T}\setminus Q}\frac{1}{|e^{it}-1|^2}\operatorname{d}\!m(e^{it}) \\
&= B_\sigma(1+K)^3\int_b^a \frac{1}{2(1-\cos t)}\operatorname{d}\!t\\&\leq B_\sigma(1+K)^3\left(\frac{1}{|1-e^{ia}|}+\frac{1}{|1-e^{ib}|}\right)\\
&\leq  2B_\sigma(1+K)^3 \frac{1}{\text{dist}(1, \mathbb{T}\setminus Q)},
\end{align*}
concluding the proof.
\end{proof}
With similar techniques and under additional conditions (assuming for instance that the arc $(\xi_i^-,\xi_i^+)$ is contained in $Q$), one could also prove a reverse estimate to Lemma \ref{L:integralestimate}. However, we will not need it in this work.

\section{Proof of Theorem \ref{T:main2}} \label{S:proofT2}
Here we present the proof of Theorem \ref{T:main2}, that is essential for establishing Theorem \ref{T:main1}. This proof relies on an application of Tolsa's boundedness criterion to the operator $\mathcal{C}_\sigma$. We recall that our definition of $\mathcal{C}_\sigma$ corresponds to a truncated Cauchy transform; see \eqref{E:Cauchydef}.

\begin{theorem}[Tolsa, Theorem 1 of \cite{Tolsa2001}]\label{Tolsa}
The Cauchy transform $\mathcal{C}_\sigma$ is bounded on $L^2(\sigma)$ if and only if there exists a positive constant $C$ such that, for every arc $Q \subseteq \mathbb{T}$,
\begin{equation}\label{E:Tolsacondition}
\|\mathcal{C}_\sigma\chi_Q\|_{L^2(\sigma)}\leq C \sigma(Q)^{1/2}.
\end{equation}
\end{theorem}
We point out that the arcs $Q$ that we are considering are free to be open, closed or semi-open.

In order to prove Theorem \ref{T:main2}, we thus need to establish that the conditions appearing in Theorem \ref{T:Bessonov} for one-component inner functions yield \eqref{E:Tolsacondition}. We emphasize that the fifth condition of that theorem is a $T(1)$-type condition. In certain situations, it is known that a $T(1)$ condition already implies a condition of type \eqref{E:Tolsacondition}; see for instance \cite{TolsaMathAnn2001} where the underlying measure satisfies a polynomial growth condition. However, to the best of our knowledge, such a result is not known for {\em purely atomic measures} that we need in this study. Moreover, our proof relies also on the other two conditions appearing in Theorem \ref{T:Bessonov}.

\begin{proof}[Proof of Theorem \ref{T:main2}]
If $\sigma(Q)=0$, both sides of \eqref{E:Tolsacondition} are zero. 

Let $Q$ be an arc that contains only one atom of $\sigma$, namely $\xi_i$. Thus,
\[
\|\mathcal{C}_\sigma\chi_Q\|_{L^2(\sigma)}^2=\sum_{j} \sigma_{\xi_j}\left|\sum_{n\neq j, \xi_n \in Q}\frac{\sigma_{\xi_n}}{1
-\xi_j\overline{\xi_n}}\right|^2
=\sum_{j\neq i}\sigma_{\xi_j}\left|\frac{\sigma_{\xi_i}}{1-\xi_j\overline{\xi_i}}\right|^2.
\]
Put $T=(\xi^-_i,\xi^+_i)$. Then, by the above calculation, we have
\begin{align*}
\|\mathcal{C}_\sigma\chi_Q\|_{L^2(\sigma)}^2
&= \sigma_{\xi_i}^2 \sum_{j\neq i}\sigma_{\xi_j}\left|\frac{1}{1-\xi_j\overline{\xi_i}}\right|^2\\
&= \sigma_{\xi_i}^2\int_{\mathbb{T}\setminus T} \frac{1}{|\xi-\xi_i|^2}\operatorname{d}\!\sigma(\xi)\\
&\lesssim  \sigma_{\xi_i}^2\text{max}\left(\frac{1}{|\xi_i-\xi^+_i|},\frac{1}{|\xi_i-\xi^-_i|}\right)\\
&\lesssim  \sigma_{\xi_i}=\sigma(Q),
\end{align*}
where in the first inequality we have used Lemma \ref{L:integralestimate} on the arc $T$ and in the second one we used the fourth condition of Theorem \ref{T:Bessonov}.

From now on we consider arcs $Q=[\xi_{Q,l},\xi_{Q,r}]$, which contain at least two atoms of $\sigma$. Notice that a priori the endpoints $\xi_{Q,l},\xi_{Q,r}$ might not be Clark atoms, but since we are estimating $\sigma(Q)$, we may assume without loss of generality that they are.

We split the exterior sum in two parts
\begin{align*}
&\sum_{j} \sigma_{\xi_j}\left| \sum_{i \neq j, \xi_i \in Q}\frac{1}{1-\xi_j\overline{\xi_i}}\sigma_{\xi_i}\right|^2=I+I\!I,
\end{align*}
where
\[
I =\sum_{\xi_j \in 2Q} \sigma_{\xi_j}\left| \sum_{i \neq j, \xi_i \in Q}\frac{1}{1-\xi_j\overline{\xi_i}}\sigma_{\xi_i}\right|^2,
\quad
I\!I=\sum_{\xi_j \in \mathbb{T}\setminus 2 Q} \sigma_{\xi_j}\left| \sum_{\xi_i \in Q}\frac{1}{1-\xi_j\overline{\xi_i}}\sigma_{\xi_i}\right|^2.
\]

For the summand $I\!I$, by the Cauchy-Schwarz inequality and Tonelli's Theorem, we have
\begin{align*}
\sum_{\xi_j \in \mathbb{T}\setminus 2 Q} \sigma_{\xi_j}\left| \sum_{ \xi_i \in Q}\frac{1}{1-\xi_j\overline{\xi_i}}\sigma_{\xi_i}\right|^2&\leq \sum_{\xi_j \in \mathbb{T}\setminus 2 Q} \sigma_{\xi_j}\sum_{ \xi_i \in Q}\frac{\sigma_{\xi_i}}{|\xi_i-\xi_j|^2}\sigma(Q)\\
& =\sigma(Q) \sum_{ \xi_i \in Q} \left(\sum_{\xi_j \in \mathbb{T}\setminus 2 Q} \frac{\sigma_{\xi_j}}{|\xi_i-\xi_j|^2}\right)\sigma_{\xi_i}.
\end{align*}

By Lemma \ref{L:integralestimate}, we have that for every $\xi_i\in Q$,
\[\sum_{\xi_j \in \mathbb{T}\setminus 2 Q} \frac{\sigma_{\xi_j}}{|\xi_i-\xi_j|^2}=\int_{\mathbb{T}\setminus 2Q}\frac{1}{|\xi_i-\xi|^2}\operatorname{d}\!\sigma(\xi)\lesssim \frac{1}{\operatorname{dist}(\xi_i,\mathbb{T}\setminus2Q)}\leq \frac{1}{|Q|/2},\]
where the constants involved only depend on $u$, and not on $\xi_i$ or $Q$.

It follows that
\begin{align*}
I\!I & \leq \sigma(Q) \sum_{ \xi_i \in Q} \left(\sum_{\xi_j \in \mathbb{T}\setminus 2Q} \frac{\sigma_{\xi_j}}{|\xi_i-\xi_j|^2}\right)\sigma_{\xi_i} \\
&\lesssim \sigma(Q) \frac{2}{|Q|}\sum_{\xi_i\in Q}\sigma_{\xi_i}\\
&\lesssim \sigma(Q),
\end{align*}
where in the third inequality Lemma \ref{Lemma misura} is used. 

For the summand $I$, we write
\begin{align*}
&\sum_{\xi_j \in 2Q} \sigma_{\xi_j}\left| \sum_{i \neq j, \xi_i \in Q}\frac{1}{1-\xi_j\overline{\xi_i}}\sigma_{\xi_i}\right|^2\leq I_a+I_b,
\end{align*}
where
\begin{align*}
I_a = 2\sum_{\xi_j \in 2Q} \sigma_{\xi_j}\left| \sum_{i \neq j, \xi_i \in Q}\left( \frac{1}{1-\xi_j\overline{\xi_i}}-\frac{1}{1-t_j\overline{\xi_i}}\right) \sigma_{\xi_i}\right|^2
\end{align*}
and
\begin{align*}
I_b=2\sum_{\xi_j \in 2Q} \sigma_{\xi_j}\left| \sum_{i \neq j, \xi_i \in Q}\frac{1}{1-t_j\overline{\xi_i}}\sigma_{\xi_i}\right|^2,
\end{align*}
with $u(t_j)=e^{2\pi i \frac{1}{N}}$ and $t_j\in(\xi_j,\xi_j^+)$. Hence
\begin{equation}\label{TO use}
|\xi_j-t_j|\asymp \frac{1}{|u'(\xi_j)|}\asymp \frac{1}{|u'(t_j)|},
\end{equation}
according to Lemma \ref{L:feichtinger}. Then, since $u(\xi_j)=1$, we have that
\begin{align*}
\frac{I_b}{2}&=\frac{1}{| 1-e^{2\pi i \frac{1}{N}}|^2} \sum_{\xi_j \in 2Q} \sigma_{\xi_j}| 1-u(t_j)|^2 \left| \sum_{i \neq j, \xi_i \in Q}\frac{1}{1-t_j\overline{\xi_i}}\sigma_{\xi_i}\right|^2\\
& \lesssim   \sum_{\xi_j \in 2Q} \sigma_{\xi_j}\left\vert 1-u(t_j)\overline{u(\xi_j)}\right\vert^2 \left| \sum_{\xi_i \in Q}\frac{1}{1-t_j\overline{\xi_i}}\sigma_{\xi_i}\right|^2 + \sum_{\xi_j \in 2Q} \sigma_{\xi_j} \left| \frac{1}{1-t_j\overline{\xi_j}}\sigma_{\xi_j}\right|^2\\
&= \sum_{\xi_j \in 2Q} \sigma_{\xi_j}|W_\sigma\chi_Q(t_j)|^2+ \sum_{\xi_j \in 2Q} \sigma_{\xi_j} \left| \frac{1}{1-t_j\overline{\xi_j}}\sigma_{\xi_j}\right|^2.
\end{align*}
In the first inequality we used the fact that $|1-e^{2\pi i \frac{1}{N}}|$ is a fixed constant. Due to Lemma \ref{Lemma misura}, we have the uniform estimate for the atoms of the two different Clark measures $\sigma_{\xi_j}\asymp\sigma_{t_j}^\frac{1}{N}$. Then
\[
\sum_{\xi_j \in 2Q} \sigma_{\xi_j}|W_\sigma\chi_Q(t_j)|^2\lesssim\int_\mathbb{T} \left\vert W_\sigma\chi_Q(t)\right\vert^2 \operatorname{d}\!\sigma^{\frac{1}{N}}(t) = \|W_\sigma \chi_Q\|_{H^2}^2=\sigma(Q).
\]
Also, by \eqref{TO use},
\[
\sum_{\xi_j \in 2Q} \sigma_{\xi_j} \left| \frac{1}{1-t_j\overline{\xi_j}}\sigma_{\xi_j}\right|^2
\lesssim \sum_{\xi_j \in 2Q}  \sigma_{\xi_j}^3|u'(\xi_j)|^2 = \sigma(2Q).
\]
It follows that
\begin{align*}
\frac{I_b}{2}  \lesssim\sigma(Q)+ \sigma(2Q) \lesssim \sigma(Q).
\end{align*}

Finally, we observe that
\begin{align*}
\frac{I_a}{2}&=\sum_{\xi_j \in 2Q} \sigma_{\xi_j}\left| \sum_{i \neq j, \xi_i \in Q}\left( \frac{1}{1-\xi_j\overline{\xi_i}}-\frac{1}{1-t_j\overline{\xi_i}}\right)\sigma_{\xi_i}\right|^2\\
&\leq \sum_{\xi_j \in 2Q} \sigma_{\xi_j}|\xi_j-t_j|^2\left( \sum_{i \neq j, \xi_i \in Q} \frac{1}{|1-\xi_j\overline{\xi_i}||1-t_j\overline{\xi_i}|}\sigma_{\xi_i}\right)^2\\
& \leq \sum_{\xi_j \in 2Q} \sigma_{\xi_j}|\xi_j-t_j|^2 \bigg(\sup_{i \neq j, \xi_i \in Q}\left\vert \frac{1-t_j\overline{\xi_i}}{1-\xi_j\overline{\xi_i}}\right\vert\bigg)^2\left( \sum_{i \neq j, \xi_i \in Q} \frac{1}{|1-t_j\overline{\xi_i}|^2}\sigma_{\xi_i}\right)^2.
\end{align*}
Again, using the fact that $|1-u(t_j)\overline{u(\xi_i)}|$ is a fixed constant, so that
\begin{align*}
\sum_{i \neq j, \xi_i \in Q} \frac{1}{|1-t_j\overline{\xi_i}|^2}\sigma_{\xi_i}
&= c\sum_{i \neq j, \xi_i \in Q} \frac{|1-u(t_j)\overline{u(\xi_i)}|^2}{|1-t_j\overline{\xi_i}|^2}\sigma_{\xi_i}\\
&\le c \|W_\sigma^{-1} k_{t_j}^u\|^2_{L^2(\sigma)}=c\|k_{t_j}^u\|_{H^2}^2,
\end{align*}
we get
\begin{align*}
\frac{I_a}{2}&\lesssim \sum_{\xi_j \in 2Q} \sigma_{\xi_j}|\xi_j-t_j|^2 \|k^{u}_{t_j}\|_{H^2}^4 \sup_{i \neq j, \xi_i \in Q}\left\vert \frac{1-t_j\overline{\xi_i}}{1-\xi_j\overline{\xi_i}}\right\vert^2\\
&\lesssim \left( \sup_{\xi_j \in 2Q} |\xi_j-t_j|^2 |u'(t_j)|^2 \right) \left( \sup_{\xi_j \in 2Q} \sup_{i \neq j, \xi_i \in Q}\left\vert \frac{1-t_j\overline{\xi_i}}{1-\xi_j\overline{\xi_i}}\right\vert^2 \right)\sigma(2Q).
\end{align*}
Now, on the one hand, by \eqref{TO use}
\[\left( \sup_{\xi_j \in 2Q} |\xi_j-t_j|^2 |u'(t_j)|^2\right)\lesssim 1.\]
On the other hand, for every $\xi_j$ and every $\xi_i$ such that $|\xi_i-t_j|\geq |\xi_i-\xi_j|$, we have that
\begin{align*}
\sup_{i \neq j, \xi_i \in Q}\left\vert \frac{\xi_i-t_j}{\xi_i-\xi_j}\right\vert&\leq 1+ \sup_{i \neq j, \xi_i \in Q}\left\vert \frac{\xi_j-t_j}{\xi_i-\xi_j}\right\vert\leq 1+ \sup_{i \neq j, \xi_i \in Q}\left\vert \frac{\xi_j-\xi_{j}^+}{\xi_i-\xi_j}\right\vert\\
   &\leq 1+ \frac{B_\sigma}{A_\sigma} \sup_{i \neq j, \xi_i \in Q}\left\vert \frac{\xi_j-\xi_{j}^-}{\xi_i-\xi_j}\right\vert\lesssim 1,
\end{align*}
uniformly in $j$, using the definition of neighboring atoms. We conclude that $I_a\lesssim \sigma(2Q)\lesssim\sigma(Q)$. Hence, we get \eqref{Tolsa} for every arc $Q$. By applying Theorem \ref{Tolsa}, we conclude the proof.
\end{proof}

We emphasize that the relationship between the inner function $u$ and the Cauchy transform $\mathcal{C}_\sigma$ is far from being fully understood. In this regard, we conclude this section with an open problem.

\begin{open} Let $\sigma$ be the Clark measure associated with $u$. Is it true that if $\mathcal{C}_\sigma$ is bounded on $L^2(\sigma)$, then $u$ is a one-component inner function?
\end{open}

If $\sigma$ is not singular, we refer the reader to \cite[Theorem 1.7]{Bellavita2022O} for the analogue of Theorem \ref{T:Bessonov}.

\section{Proof of Theorem \ref{T:main1}: Necessary conditions} \label{S:proofT1measure}
In this section we start the proof of Theorem \ref{T:main1}: we show that, under its hypotheses, if we have the set identity $\mathcal{H}(b)= \mathcal{D_\mu}$, then $\mu$ is necessarily defined as in \eqref{E:main1expressionmu} and it satisfies \eqref{E:main1atoms}, \eqref{E:main1potential}. We highlight that whenever $\mathcal{H}(b)= \mathcal{D_\mu}$, by the closed graph theorem, we also have an equivalence of the norms.

The interplay between the Pythagorean mate $a$ and the potential $V_\mu$ defined in \eqref{E:potentialdef} is crucial in this study. We will use the following obvious estimates for $V_\mu$:
\begin{equation} \label{E:Properties of Vmu}
\frac{\mu(\mathbb{T})}{(1+|w|)^2}\leq V_\mu(w)\leq \frac{\mu(\mathbb{T})}{\text{dist}(w, \operatorname{supp}\mu)^2}, \qquad w\in\mathbb{D}.
\end{equation}

The following lemma is an adaptation of results that first appeared in \cite{costara2013}. We present them here in a form suited to our purposes and include a proof for the reader’s convenience. Recall that $X \hookrightarrow Y$ means that $X$ is boundedly embedded in $Y$.

\begin{lemma}\label{L:Costara}
Let $(b,a)$ be a Pythagorean pair and $\mu$ a measure on $\mathbb{T}$. Then the following statements hold.
\begin{enumerate}[(i)]
\item If $\mathcal{H}(b)\hookrightarrow\mathcal{D}_\mu$, then
\[
|a(w)|^2 V_\mu(w) \lesssim |a(w)|^2+|b(w)|^2,
\]
uniformly for $w \in \mathbb{D}$.
\item If $\mathcal{D}_\mu \hookrightarrow \mathcal{H}(b)$, then
\[
|a(w)|^2+|b(w)|^2 \lesssim |a(w)|^2 V_\mu(w)
\]
uniformly for $w \in \mathbb{D}$.
\end{enumerate}
\end{lemma}

\begin{proof}
In \cite{costara2013}, the authors computed the $\mathcal{H}(b)$ and $\mathcal{D}_\mu$ norms of the Cauchy-Szeg\"o kernels
\[
c_w(z)= \frac{1}{1-\overline{w}z}, \qquad w,z\in\mathbb{D}.
\]
In particular, they showed that
\begin{equation}
\|c_w\|_{\mathcal{H}(b)}^2 = \frac{1+|b(w)/a(w)|^2}{1-|w|^2},\quad
\|c_w\|_{\mathcal{D}_\mu}^2 =\frac{1+|w|^2V_\mu(w)}{1-|w|^2}.
\end{equation}
Then, whenever $\mathcal{H}(b)\hookrightarrow\mathcal{D}_\mu$, it follows that
\[
|a(w)|^2(1+|w|^2 V_\mu(w)) \lesssim |a(w)|^2+|b(w)|^2,
\]
and when $\mathcal{D}_\mu\hookrightarrow \mathcal{H}(b),$
\[
|a(w)|^2+|b(w)|^2 \lesssim |a(w)|^2(1+|w|^2 V_\mu(w)).
\]

To conclude, it is enough to show that $V_\mu(w)\asymp (1+|w|^2V_\mu(w))$ uniformly in $\mathbb{D}$. By \eqref{E:Properties of Vmu},
\begin{equation*}
V_\mu(w)\geq \frac{\mu(\mathbb{T})}{4}, \qquad w\in\mathbb{C},
\end{equation*}
so that
\[
1+|w|^2V_\mu(w)\leq V_\mu(w)\left(\frac{4}{\mu(\mathbb{T})}+1\right).
\]
Splitting the disk into two subsets $|w|\leq1/2$ and $|w|>1/2$, and with the second inequality of \eqref{E:Properties of Vmu} in mind for the case $|w|\leq 1/2$, we see that
\[
V_\mu(w)\leq \max\left(4\mu(\mathbb{T}),4|w|^2V_\mu(w) \right)\leq \max\left(4\mu(\mathbb{T}),4 \right)\left( 1+|w|^2V_\mu(w)\right).
\]
\end{proof}

Notice that Lemma \ref{L:Costara} directly implies the necessity of condition \eqref{E:main1potential}. In the rest of this section we prove that $\mu$ is necessarily of the form \eqref{E:main1expressionmu} and satisfies \eqref{E:main1atoms}. For this, we recall first that in \cite{bellavita2023embedding}, the authors established some relationship between the boundary spectrum of $b$ and the measure $\mu$. This result was further generalized in \cite[Theorem 3.3]{Eugenionext}.

\begin{lemma}[Dellepiane--Peloso--Tabacco, \cite{Eugenionext}]\label{Spectrum}
Let $\mu$ be a finite positive Borel measure on $\mathbb{T}$ and let $b$ be a bounded analytic function with $\|b\|_{H^\infty}=1$.
If the embedding $\mathcal{H}(b) \hookrightarrow \mathcal{D}_\mu$ holds, then $\mu(\rho(b))=0$.
\end{lemma}

Studying the equality $\mathcal{H}(b)=\mathcal{D}_\mu$ requires the introduction of the boundary zero set
\[
\mathcal{Z}_\mathbb{T}(a)=\{\lambda \in \mathbb{T} \ \colon \ \lim_{r \to 1}a(r\lambda)=0\}.
\] 
In what follows, we will just write $\mathcal{Z}(a)$, omitting $\mathbb{T}$ from the notation.

\begin{theorem}\label{Tza}
Let $(b,a)$ be a Corona pair, and let $\mu$ be a finite positive Borel measure on $\mathbb{T}$. If $\mathcal{H}(b)= \mathcal{D_\mu}$, then \[\operatorname{supp} (\mu) = \overline{\mathcal{Z}(a)}.\]
\end{theorem}

\begin{proof}
Let $\zeta$ be in $\mathbb{T}\setminus\mathcal{Z}(a)$. Then, since $\limsup_{r\to 1} |a(r\zeta)|>0$, there exist $\epsilon>0$ and a sequence $(r_n)_n$ converging to $1$ such that $\lim_{n}|a(r_n\zeta)|>\epsilon$. From Lemma \ref{L:Costara}(i) and Fatou's Lemma, we note that, if $\mathcal{H}(b)=\mathcal{D}_\mu$, then
\[
V_\mu(\zeta) \leq \liminf_{n} V_\mu(r_n\zeta )\lesssim \lim_{n}\frac{1}{|a(r_n\zeta)|^2}<\infty.
\]
Consequently,
\[
\mathbb{T}\setminus \mathcal{Z}(a)\subseteq \{\zeta \in \mathbb{T}\ \colon V_\mu(\zeta)<\infty\}.
\]
Since $V_\mu = \infty$ $\mu$-a.e., then $\mu(\mathbb{T}\setminus \mathcal{Z}(a)) = 0$, implying that $\mathcal{Z}(a)$ is a carrier for $\mu$.
Therefore, by the definition of support, we have also that $\operatorname{supp}(\mu)\subseteq \overline{\mathcal{Z}(a)}$. Notice that for this set inclusion, it was not needed that $(b,a)$ forms a Corona pair.

For the other inclusion, we consider $\zeta \in \mathbb{T}\setminus \operatorname{supp}(\mu)$. By \eqref{E:Properties of Vmu}, Lemma \ref{L:Costara}(ii) and the assumption that $(b,a)$ forms a Corona pair,
\begin{equation}\label{Zeq1}
\operatorname{dist}(w, \operatorname{supp}(\mu))^2 \leq \frac{\mu(\mathbb{T})}{V_\mu(w)} \asymp \frac{\mu(\mathbb{T}) |a(w)|^2}{|a(w)|^2+|b(w)|^2} \lesssim |a(w)|^2
\end{equation}
for every $w\in\mathbb{D}$. Then
\[
0< \liminf_{r\to1}\operatorname{dist}(r\zeta, \operatorname{supp}(\mu))\lesssim \liminf_{r\to1} |a(r\zeta)|,
\]
proving that $\zeta \notin \mathcal{Z}(a)$. Hence, $ \mathbb{T}\setminus \operatorname{supp}(\mu) \subset \mathbb{T}\setminus \mathcal{Z}(a)$, which is equivalent to $\mathcal{Z}(a) \subset \operatorname{supp}(\mu)$.
\end{proof}

We are now in a position to prove the necessity of \eqref{E:main1expressionmu} and \eqref{E:main1atoms}. Since $b=(1+u)/2$ and $a=\gamma(1-u)/2$, due to Theorem \ref{Tza}, if $\mathcal{H}(b)=\mathcal{D}_\mu$, then
\[
\operatorname{supp}(\mu) = \overline{\mathcal{Z}(a)} = \overline{\{\zeta\in\mathbb{T} \colon \lim_{r\to 1} u(r\zeta)=1\}}.
\]
Using the one-component condition and Theorem 1.11 of \cite{Aleksandrov2000}, all the points that satisfy $\lim_{r\to 1} u(r\zeta)=1$ lie outside the spectrum $\rho(u)$. Then, the function $u$ is analytic in every such point, and in particular they are all atoms for $\sigma=\sigma^0$, the Clark measure associated to $u$. It follows that
\[
\operatorname{supp}(\mu) = \overline{\mathcal{Z}(a)} = \overline{\{\zeta\in\mathbb{T}\setminus\rho(u) \colon \lim_{r\to 1} u(r\zeta)=1\}}=\operatorname{supp}(\sigma).
\]
Since $\tau(\sigma)=\rho(u)\subset \rho(b)$ and $\mu(\rho(b))=0$, we have also that $\mu(\tau(\sigma))=0$. Consequently, $\mu$ is a purely atomic measure of the form
\begin{equation}\label{E:expressionmugeneralmasses}
  \mu=\sum_{n}\mu_n \delta_{\zeta_n},
\end{equation}
which yields \eqref{E:main1expressionmu}. Notice that $\{\zeta_n\}_n$ consists of isolated points and $\sum_n \mu_n = \mu(\mathbb{T})<\infty$. The following lemma is needed in order to estimate the masses $\mu_n$.

\begin{lemma}\label{L:Lemma5.1(1)}
Let $u$ be a one-component inner function and $\mu$ a measure having the same support as the zero Clark measure $\sigma$ of $u$, as in \eqref{E:expressionmugeneralmasses}. Then, for every atom $\zeta_k\in a(\mu)$,
\[
\lim_{z\to\zeta_k}|1-u(z)|^2V_\mu(z)=|u'(\zeta_k)|^2\mu_k.
\]
In particular, the function $|1-u|^2V_\mu$ admits a continuous extension at every atom, and hence to $\mathbb{T}\setminus\rho(u)$.
\end{lemma}

\begin{proof}
We write
\[
|1-u(z)|^2V_\mu(z) =   \mu_k\frac{|1-u(z)|^2}{|z-\zeta_k|^2} + \sum_{n\neq k}\mu_n\frac{|1-u(z)|^2}{|z-\zeta_n|^2}.
\]
Since $\zeta_k$ is an isolated atom, $\delta=\operatorname{dist}(\zeta_k,\operatorname{supp}(\sigma)\setminus\{\zeta_k\})>0.$ If $|z-\zeta_k|<\delta/2$, for every $n\neq k$ we have that
\[
\mu_n\frac{|1-u(z)|^2}{|z-\zeta_n|^2} \leq \mu_n\frac{|1-u(z)|^2}{(|z-\zeta_k|-|\zeta_k-\zeta_n|)^2}\leq \frac{16}{\delta^2} \mu_n,
\]
which is a summable sequence. Thus, by the Lebesgue Dominated Convergence Theorem, it follows that
\[
\lim_{z\to\zeta_k}|1-u(z)|^2V_\mu(z)=\mu_k |u'(\zeta_k)|^2 +\sum_{n\neq k}\lim_{z\to\zeta} \mu_n\frac{|1-u(z)|^2}{|z-\zeta_n|^2}=\mu_k |u'(\zeta_k)|^2.
\]
\end{proof}

To finish, note that by Lemma \ref{L:Costara}, since $(b,a)$ is a Corona pair and $\mathcal{H}(b)=\mathcal{D}_\mu$, we have that $|a|^2V_\mu \asymp 1$ and so
\[
\mu_n\asymp \frac{1}{|u'(\zeta_n)|^2}
\]
uniformly in $n$, which establishes \eqref{E:main1atoms}.

\section{Proof of Theorem \ref{T:main1}: Necessary and sufficient conditions for $\mathcal{H}(b)\hookrightarrow \mathcal{D_\mu}$} \label{S:proofT1embHb}
In this section, we focus on the inclusion $\mathcal{H}(b)\hookrightarrow \mathcal{D_\mu}$, exploring the conditions under which this embedding holds. We will suppose that $\mu$ is given by
\begin{equation} \label{E:defmu0}
\mu=\sum_{n}\mu_n\delta_{\zeta_n}.
\end{equation}
The following lemma will be used several times, when comparing $|a|^2=|1-u|^2/4$ and $V_\mu$.

\begin{lemma} \label{L:Lemmaequiv}
Let $u$ be a one-component inner function, let $\{\zeta_n\}_n$ its Clark points such that $u(\zeta_n)=1$, and let $\mu$ be defined as in \eqref{E:defmu0}. Then we have that
\[
\sup_{\xi \in \mathbb{T}\setminus \rho(u)} \limsup_{z\to\xi}|1-u(z)|^2V_\mu(z) <\infty \iff \sup_m \sum_{n\neq m} \frac{\mu_n}{|\zeta_n-\zeta_m|^2}<\infty.
\]
\end{lemma}

\begin{proof}
We begin with the implication ‘‘$\Longleftarrow$". First, we show that the assumption implies that $\sup_n \mu_n |u'(\zeta_n)|^2<\infty$. Indeed, by Theorem \ref{T:Bessonov},
\[
\mu_n |u'(\zeta_n)|^2 \lesssim \frac{\mu_n}{|\zeta_n-\zeta_n^+|^2}\leq \sup_m \sum_{n\neq m} \frac{\mu_n}{|\zeta_n-\zeta_m|^2}.
\]
Thus, in light of Lemma \ref{L:Lemma5.1(1)}, the assertion holds when $\xi$ is a Clark atom.

Now, suppose that $\xi$ is not an atom, and let $\zeta_{N_1},\zeta_{N_2}$ be the two atoms such that $\xi$ belongs to the arc with endpoints $\zeta_{N_1},\zeta_{N_2}$, and such that $\zeta_{N_1},\zeta_{N_2}$ are neighbors. Notice that $\xi\notin a(\sigma)\cup\tau(\sigma)=\operatorname{supp}(\sigma)$, and the function $|a|^2V_\mu$ extends continuously to $\xi$. Then
\begin{align*}
|1-u(\xi)|^2 V_\mu(\xi) &= \sum_{n} \mu_n\frac{|1-u(\xi)|^2}{|\zeta_n-\xi|^2} \\
&\lesssim\sum_{n\neq N_1,N_2} \mu_n\frac{|1-u(\xi)|^2}{|\zeta_n-\xi|^2} + \frac{|k_{\zeta_{N_1}}(\xi)|^2}{|u'(\zeta_{N_1})|^2}+\frac{|k_{\zeta_{N_2}}(\xi)|^2}{|u'(\zeta_{N_2})|^2}  \\
&\leq 2C_u+4\sum_{n\neq N_1,N_2} \frac{\mu_n}{|\zeta_n-\xi|^2},
\end{align*}
where the constant $C_u$ introduced in \eqref{E:Relation norms kernels one-component}.  By the definition of $\zeta_{N_1},\zeta_{N_2}$, we have that, for every $n\neq N_1,N_2$,
\[
|\xi-\zeta_n|\geq \min\{|\zeta_{N_1}-\zeta_n|,|\zeta_{N_2}-\zeta_n|\}.
\]
It follows that
\begin{align*}
|1-u(\xi)|^2  V_\mu(\xi) &\lesssim 1+\sum_{n\neq N_1,N_2} \frac{\mu_n}{\min\{|\zeta_{N_1}-\zeta_n|,|\zeta_{N_2}-\zeta_n|\}^2} \\
&\leq 1+2 \sup_m \sum_{n\neq m} \frac{\mu_n}{|\zeta_m-\zeta_n|^2},
\end{align*}
where the last inequality follows from splitting the sum into two parts, separating the atoms $\zeta_n$ such that $|\zeta_{N_1}-\zeta_n|\geq |\zeta_{N_2}-\zeta_n|$ from those which satisfy the reverse inequality. Consequently,
\[
\sup_{\xi \in \mathbb{T}\setminus \rho(u)} \limsup_{z\to\xi}|1-u(z)|^2V_\mu(z) <\infty.
\]

We now establish the impplication ‘‘$\implies$". For every $m$ fixed, we choose the two atoms $\eta_1,\eta_2$ of the Clark measure $\sigma^{1/2}$ of $u$ such that $\zeta_m$ belongs to the open arc with endpoints $\eta_1,\eta_2$ and such that $\eta_1,\eta_2$ are neighbors. Observe that $u(\eta_1)=u(\eta_2)=-1$. Since obviously the atoms of $\sigma$ and $\sigma^{1/2}$ are intertwined, arguing as in the previous case, for every $n\neq m$ we have that
\[
|\zeta_n-\zeta_m|\geq \min\{|\zeta_n-\eta_1|,|\zeta_n-\eta_2|\}.
\]
It follows that
\begin{align*}
\sum_{n\neq m} \frac{\mu_n}{|\zeta_n-\zeta_m|^2} &\leq \sum_{n\neq m}\frac{\mu_n}{\min\{|\zeta_n-\eta_1|,|\zeta_n-\eta_2|\}^2} \\
&\leq \sum_{n\neq m} \frac{\mu_n}{|\zeta_n-\eta_1|^2} + \sum_{n\neq m} \frac{\mu_n}{|\zeta_n-\eta_2|^2} \\
&=\sum_{n\neq m}\mu_n\frac{|1-u(\eta_1)|^2}{4|\zeta_n-\eta_1|^2} + \sum_{n\neq m}\mu_n\frac{|1-u(\eta_2)|^2}{4|\zeta_n-\eta_2|^2} \\
&\leq \frac{1}{2} \sup_{\mathbb{T}\setminus\rho(u)} |1-u|^2 V_\mu,
\end{align*}
since the atoms of the Clark measure $\sigma^{1/2}$ do not belong to $\rho(u)$.
\end{proof}

We will now discuss necessary and sufficient conditions for the embedding $\mathcal{H}(b)\hookrightarrow\mathcal{D}_\mu$, which is a first step to the sufficiency part of Theorem \ref{T:main1}. We will need the difference quotient operator that was studied in \cite{bellavita2024spectralanalysisdifferencequotient}. For a point $\zeta\in\mathbb{T}\setminus\rho(u)$, the mapping $Q_\zeta^u\colon K_u\to K_u$ that assigns to each $f\in K_u$ the function
\[Q_\zeta^uf(z)=\frac{f(z)-f(\zeta)}{z-\zeta}, \qquad z\in\mathbb{D},\]
defines a bounded operator on $K_u$. It plays a crucial role in our analysis since $\mathcal{D}_\zeta(f)=\|Q_\zeta^u f\|_{H^2}^2.$

\begin{theorem} \label{T:mainTemb}
Let $u$ be a one-component inner function, $b=(1+u)/2$, and $\mu = \sum_{n}\mu_n\delta_{\zeta_n}$ having the same support as the Clark measure of $u$. Then      $\mathcal{H}(b)\hookrightarrow \mathcal{D}_\mu$ if and only if
\begin{equation}\label{thmitem1}
    \sup_{z \in \mathbb{D}} |a(z)|^2 V_\mu(z) <\infty.
\end{equation}
\end{theorem}

\begin{proof}
Lemma \ref{L:Costara}(i) shows that the condition on the potential is necessary. Now assume that $\sup_{z \in \mathbb{D}} |a(z)|^2 V_\mu(z) <\infty$. The function $|a|^2 V_\mu$ admits a continuous extension at every point $\lambda\in\mathbb{T}\setminus\rho(u)$; this follows from Lemma \ref{L:Lemma5.1(1)} and the properties of $a$ and $V_\mu$. Thus,
\[
\sup_{ \lambda \in \mathbb{T}\setminus\rho(u)} |a(\lambda)|^2 V_\mu(\lambda) \leq \sup_\mathbb{D}|a|^2 V_\mu <\infty.
\]
In particular, by Lemma \ref{L:Lemma5.1(1)}, we have that $\mu_n\lesssim |u'(\zeta_n)|^{-2}$, uniformly in $n$.

Since, according to Lemma \ref{L:FricainGrivaux}, $\mathcal{H}(b)=aH^2\oplus_b K_u$, we have to show that the two spaces embed separately into $\mathcal{D}_\mu$. We first show that $aH^2 \hookrightarrow \mathcal{D}_\mu$.
Considering $f=ag \in aH^2$,
\[
\mathcal{D}_\mu(ag) =\sum_{j} \mu_j \mathcal{D}_{\zeta_j}(ag) = \sum_j \mu_j\int_{\mathbb{T}}\left\vert \frac{a(\lambda)g(\lambda)-a(\zeta_j)g(\zeta_j)}{\lambda-\zeta_j}\right\vert^2\operatorname{d}\!m(\lambda).
\]
Since $a(\zeta_j)g(\zeta_j)=0$, by \cite[Lemma 2.3]{FricainGrivaux},  we have that
\begin{align*}
\mathcal{D}_\mu(ag)&=\int_{\mathbb{T}}|g(\lambda)|^2 |a(\lambda)|^2\left( \sum_j  \frac{\mu_j}{|\lambda-\zeta_j|^2}\right)\operatorname{d}\!m(\lambda)\\
&= \int_{\mathbb{T}\setminus\rho(u)}|g(\lambda)|^2 |a(\lambda)|^2V_\mu(\lambda)\operatorname{d}\!m(\lambda) \\
&\leq \|g\|_{H^2}^2\sup_{\lambda \in \mathbb{T}\setminus \rho(u)}|a(\lambda)|^2V_\mu(\lambda),
\end{align*}
where we have used the fact that $|\rho(u)|=0$. Hence, according to Lemma \ref{L:FricainGrivaux},
\begin{align*}
\|f\|_{\mathcal{D}_\mu}^2&=\|ag\|_{H^2}^2+\mathcal{D}_\mu(ag)\\
&\leq \Big( \|a\|_{H^\infty}^2+ \sup_{ \lambda \in \mathbb{T}\setminus\rho(u)} |a(\lambda)|^2 V_\mu(\lambda) \Big) \|g\|_{H^2}^2\asymp\|f\|_{\mathcal{H}(b)}^2.
\end{align*}

We next show that $K_u \hookrightarrow \mathcal{D}_\mu$. We write $\widetilde{k_j}$ to denote the normalized reproducing kernel of $K_u$ associated to $\zeta_j$. Since $\{\widetilde{k_j}\}_j$ forms an orthonormal basis of $K_u$, for every $f\in K_u$, we have that
\[
f(\lambda) = \sum_j \gamma_j \widetilde{k_j}(\lambda)= \sum_j \frac{f(\zeta_j)}{|u'(\zeta_j)|}k_{\zeta_j}^u(\lambda), \quad \text{for}\,\,m\text{-a.e.}\,\,\lambda\in\mathbb{T},
\]
where $\gamma_j = \langle f,\widetilde{k_j}\rangle_{H^2}$. Observe that, for each $n$,
\begin{align*}
\mathcal{D}_{\zeta_n}(f) =  \int_{\mathbb{T}} |Q_{\zeta_n}^uf(\lambda)|^2 \operatorname{d}\!m(\lambda).
\end{align*}
Thus,
\begin{align*}
\mathcal{D}_{\zeta_n}(f)
&= \int_{\mathbb{T}} \Big| Q^u_{\zeta_n}\left(\sum_{j\neq n} \frac{f(\zeta_j)}{|u'(\zeta_j)|} k_{\zeta_j}^u\right)(\lambda)+\frac{f(\zeta_n)}{|u'(\zeta_n)|} Q^u_{\zeta_n}k_{\zeta_n}^u(\lambda) \Big|^2 \operatorname{d}\!m(\lambda) \\
&\leq 2\left\|Q^u_{\zeta_n}\left(\sum_{j\neq n} \frac{f(\zeta_j)}{|u'(\zeta_j)|} k_{\zeta_j}^u\right)\right\|_{H^2}^2 + 2\frac{|f(\zeta_n)|^2}{|u'(\zeta_n)|^2} \left\|Q^u_{\zeta_n}k_{\zeta_n}^u  \right\|_{H^2}^2.
\end{align*}
Using the estimate $\|Q_{\zeta_n}^u\|\lesssim |u'(\zeta_n)|$ from \cite[Theorem 2.8]{bellavita2024spectralanalysisdifferencequotient}, where the constant involved only depends on $u$, it follows that
\begin{align}
\nonumber         \mathcal{D}_\mu(f) &=  \sum_n \mu_n \mathcal{D}_{\zeta_n}(f)\\ 
\label{E:thm32}
 &\lesssim\sum_n \mu_n \left\|Q^u_{\zeta_n}\left(\sum_{j\neq n} \frac{f(\zeta_j)}{|u'(\zeta_j)|} k_{\zeta_j}^u\right)\right\|_{H^2}^2 + \sum_n \mu_n |u'(\zeta_n)||f(\zeta_n)|^2.
\end{align}
By the Parseval identity,
\[
\sum_n \mu_n |u'(\zeta_n)||f(\zeta_n)|^2 \lesssim \sum_n\frac{|f(\zeta_n)|^2}{|u'(\zeta_n)|} = \sum_n |\gamma_n|^2 = \|f\|_{H^2}^2,
\]
so that we only have to deal with the first summand in \eqref{E:thm32}. Since for each $j$ we have the relation
\[
\frac{f(\zeta_j)}{|u'(\zeta_j)|} k_{\zeta_j}^u = \gamma_j \widetilde{k_j},
\]
by \cite[Equation 6.6]{bellavita2024spectralanalysisdifferencequotient}, the delicate norm identity
\begin{eqnarray}
&& \left\|Q^u_{\zeta_n}\!\!\left(\sum_{j\neq n} \frac{f(\zeta_j)}{|u'(\zeta_j)|} k_{\zeta_j}^u\right)\right\|_{H^2}^2 \!\!\!\! \notag\\
&=& \sum_{j\neq n} \frac{|\gamma_j|^2}{|\zeta_j-\zeta_n|^2}+\frac{1}{\|k_{\zeta_n}^u\|_{H^2}^2} \left|  \sum_{j\neq n} \gamma_j\frac{(k_{\zeta_j}^u)'(\zeta_n)}{\|k_{\zeta_j}^u\|_{H^2}}   \right|^2\!\! \label{E:mainTembeq1}
\end{eqnarray}
holds. Hence, by \eqref{thmitem1} and by Lemma \ref{L:Lemmaequiv},
\[
\sum_n \mu_n\sum_{j\neq n} \frac{|\gamma_j|^2}{|\zeta_j-\zeta_n|^2} =\sum_j |\gamma_j|^2\sum_{n\neq j} \frac{\mu_n}{|\zeta_j-\zeta_n|^2} \lesssim \|f\|_{H^2}^2.
\]

We deal with the last term in \eqref{E:mainTembeq1}. Since $u(\zeta_m)=1$ for every atom $\zeta_m$, we have that
\begin{equation*}\label{derzero}
(k_{\zeta_j}^u)'(\zeta_n)=\frac{-\overline{u(\zeta_j)}u'(\zeta_n)(1-\overline{\zeta_j}\zeta_n)+\overline{\zeta_j}(1-\overline{u(\zeta_j)}u(\zeta_n))}{(1-\overline{\zeta_j}\zeta_n)^2}=-\frac{u'(\zeta_n)}{1-\overline{\zeta_j}\zeta_n}.
\end{equation*}
It follows that
\begin{align*}
\mathcal{D}_\mu(f)&\lesssim \|f\|_{H^2}^2 + \sum_n \frac{\mu_n}{|u'(\zeta_n)|}  \left|  \sum_{j\neq n} \gamma_j\frac{u'(\zeta_n)}{\|k_{\zeta_j}^u\|_{H^2}(1-\overline{\zeta_j}\zeta_n)}   \right|^2 \\
&\lesssim\|f\|_{H^2}^2 + \sum_n \frac{1}{|u'(\zeta_n)|} \left|  \sum_{j\neq n} \frac{f(\zeta_j)}{1-\overline{\zeta_j}\zeta_n} \frac{1}{|u'(\zeta_j)|}  \right|^2 \\
&=\|f\|_{H^2}^2 + \|\mathcal{C}_\sigma f\|_{L^2(\sigma)}^2\lesssim \|f\|_{H^2}^2,
\end{align*}
where in the last line we have used Theorem \ref{T:main2} which ensures that $\mathcal{C}_\sigma$  is bounded since $u$ is one-component.
\end{proof}

\section{Proof of Theorem \ref{T:main1}: Necessary and sufficient conditions for $\mathcal{D}_\mu \hookrightarrow \mathcal{H}(b)$} \label{S:proofT1embDm}
In order to take care of the embedding $\mathcal{D}_\mu\hookrightarrow \mathcal{H}(b)$, we use the following result by Malman and Seco, that applies to Banach spaces that admit a Cauchy dual.

\begin{theorem}[Malman--Seco, Theorem C in \cite{malmanseco}] \label{T:Malmanseco}
Let $X$ be a Banach space of analytic functions on $\mathbb{D}$ in which the analytic polynomials are dense, let $(b,a)$ be a Pythagorean pair and $\phi = b/a$. The following two statements are equivalent:
\begin{enumerate}[(i)]
\item The multiplication operator $M_\phi \colon H^2 \to X^*, f\mapsto \phi f$ is bounded, where $X^*$ is the Cauchy dual of $X$.
\item We have the embedding $X\hookrightarrow \mathcal{H}(b)$.
\end{enumerate}
\end{theorem}

For the precise definition of Cauchy dual, we refer to \cite[Section 2.3]{malmanseco}. An explicit characterization of the Cauchy dual of $\mathcal{D}_\mu$ was provided by Luo \cite{luocauchy}\footnote{ We point out that our notation for $V_\mu$ differs from the one in \cite{luocauchy}.}. Given a measure $\mu$ on $\mathbb{T}$, we consider the space
\[
L^2_{a,\mu}=\{f\in \text{Hol}(\mathbb{D})\colon \int_{\mathbb{D}}|f'(z)|^2 \frac{1-|z|^2}{V_\mu(z)}\,\operatorname{d}\!A(z) <\infty\},
\]
with norm
\[
\|f\|_{L^2_{a,\mu}}^2=|f(0)|^2+\int_{\mathbb{D}}|f'(z)|^2 \frac{1-|z|^2}{V_\mu(z)}\,\operatorname{d}\!A(z).
\]
Then let $E(\mu)$ be the closure of the analytic polynomials in $L^2_{a,\mu}$. Lemma 2.2 in \cite{luocauchy} states that $E(\mu)$ is the Cauchy dual of the space~$\mathcal{D}_\mu$.

We need the following well-known result. It is indeed immediate since outer functions are cyclic in the Hardy space $H^2$, and the latter is dense in the weighted Bergman space
\[
A_1^2 =\{f\in\operatorname{Hol}(\mathbb{D}) \colon \int_{\mathbb{D}} |f(z)|^2 (1-|z|^2) \, \operatorname{d}\!A(z) <\infty\}.
\]
For a reference, see Exercise 1 of Section 7.6 in \cite{hedenmalm2012theory}.

\begin{lemma}\label{L:cyclicity}
Let $a$ be an outer function in $H^\infty$. Then $a$ is cyclic in~$A^2_1$.
\end{lemma}

Given an analytic function $g$, the Volterra-type integral operator $T_g$ is defined as the path integral
\begin{equation} \label{D:Tg}
T_g f(z) = \int_0^z g'(\zeta) f(\zeta) \operatorname{d}\!\zeta, \qquad z\in\mathbb{D},
\end{equation}
where $f\in \operatorname{Hol}(\mathbb{D})$. Hence, $T_g f$ is a primitive of $g'f$, that is, $(T_g f)'=g'f$. It is a well-known result that $T_g$ defines a bounded operator on $H^2$ if and only if $g\in \operatorname{BMOA}$; see  \cite{Aleman1995} and \cite[Section 6]{garnett}.

\begin{lemma} \label{L:phiinEm}
Let $b=(1+u)/2$ with $u$ inner, $a$ its Pythagorean mate, $\phi = b/a$, and let $\mu$ be any finite positive measure on $\mathbb{T}$. If
\[
\inf_{\mathbb{D}}|a|^2 V_\mu>0,
\]
then for every polynomial $P$ the function $P\phi$ belongs to the space $E(\mu)$.
\end{lemma}

\begin{proof}
We need to show that, given any polynomial $P$ and any $\varepsilon>0$, there exists a polynomial $R$ such that $\|P\phi - R\|_{L^2_{a,\mu}} <\varepsilon$. Set
\[
L_u = \inf_\mathbb{D}|a|^2 V_\mu >0.
\]
Using the formula \cite[Pag. 4]{hedenmalm2012theory}
\begin{equation}\label{E:formulataylorBergman}
\|f\|_{A_1^2}^2 =\sum_n  \frac{|a_n|^2}{(n+2)(n+1)},
\end{equation}
since $Pb\in H^2,$ we have that $(Pb)'\in A^2_1$. Therefore, by Lemma \ref{L:cyclicity}, since $a$ is outer, there exists a polynomial $Q$ such that 
\[
\|(Pb)'-aQ\|_{A^2_1}^2<\frac{L_u\varepsilon}{4}.
\]
We choose the polynomial $R_1$ such that $R_1'=Q$ and $R_1(0)=0$. Subsequently, we notice that
\[
P\phi a'=\frac{a'}{a}Pb=\big(T_{\log(a)} (Pb)\big)',
\]
where $T_{\log(a)}$ is the operator defined in \eqref{D:Tg}. Here, the symbol $\log(a)$ is well defined, since the function
\[
h(z)=\log\left(\frac{1-z}{2}\right),\qquad z\in\mathbb{D},
\]
belongs to $\operatorname{BMOA},$  and then so does the composition $\log(a)=h\circ u$; see \cite[Section 2]{bourdon} and \cite{MR3431581}. It follows that $T_{\log(a)} (Pb)\in H^2$. Thus, again by the formula \eqref{E:formulataylorBergman}, $P\phi a'=\big(T_{\log(a)} (Pb)\big)'\in A^2_1$. We choose another polynomial $R_2$ such that $R_2(0)=0$ and 
\[
\|P\phi a'-aR_2'\|_{A^2_1}^2<\frac{L_u\varepsilon}{4}.
\]
Now, $R=P(0)\phi(0)+R_1-R_2$ is the desired polynomial. We have that
\begin{align*}
\|P\phi-R\|_{L^2_{a,\mu}}^2 &= \int_{\mathbb{D}} |(P\phi)'(z)-R'(z)|^2 \frac{1-|z|^2}{V_\mu(z)}\,\operatorname{d}\!A(z) \\
&= \int_{\mathbb{D}} \left\vert (Pb)'(z)-\frac{P(z)b(z)a'(z)}{a(z)}-a(z)\big(Q(z)-R_2'(z)\big)\right\vert^2 \\
&\phantom{aaa} \frac{1-|z|^2}{|a(z)|^2V_\mu(z)}\,\operatorname{d}\!A(z).
\end{align*}
It follows that
\begin{align*}
\|P\phi-R\|_{L^2_{a,\mu}}^2  &\leq \frac{2}{L_u}\|(Pb)'-aQ\|_{A^2_1}^2 +\frac{2}{L_u}\|P\phi a'-aR_2'\|_{A^2_1}^2<\varepsilon.
\end{align*}
\end{proof}

Note that in the above proof we do not use the specific form $b=(1+u)/2$. We only needed that $\log(a)\in \operatorname{BMOA}$. We are now ready to prove the result on the embedding $\mathcal{D}_\mu \hookrightarrow \mathcal{H}(b)$.
We remark that the function $u$ does not need to be one-component.

\begin{theorem} \label{T:EmbDminHb}
Let $u$ be an inner function, $b=(1+u)/2$, and let $\mu$ be a measure having the same support as the Clark measure of $u$. Then the embedding $\mathcal{D}_\mu \hookrightarrow \mathcal{H}(b)$ holds if and only if
\begin{equation}\label{E:inf-pos-ymp}
\inf_{z \in \mathbb{D}}|a(z)|^2 V_\mu(z)>0.
\end{equation}
\end{theorem}

\begin{proof}
If the embedding $\mathcal{D}_\mu\hookrightarrow \mathcal{H}(b)$ holds, then, by Lemma \ref{L:Costara}(ii),
\[
|a(w)|^2+|b(w)|^2 \lesssim |a(w)|^2V_\mu(w), \qquad w\in\mathbb{D},
\]
and since $(b,a)$ is a Corona pair, we immediately conclude that \eqref{E:inf-pos-ymp} holds.

To prove that the embedding holds under the potential condition \eqref{E:inf-pos-ymp}, we use Theorem \ref{T:Malmanseco} and the characterization of the Cauchy dual of $\mathcal{D}_\mu$. We show that the multiplication operator $M_\phi$ defines a bounded operator from $H^2$ to $E(\mu)$. First, we show that it is bounded from $H^2$ to $L^2_{a,\mu}$, and then that the image $\operatorname{Ran}(M_\phi)\subseteq E(\mu)$. As before, let
\[
L_u =\inf_{\mathbb{D}}|a|^2 V_\mu >0.
\]
Then, for $f\in H^2$,
\begin{align*}
\|M_\phi f\|_{L^2_{a,\mu}}^2 &= |\phi(0)f(0)|^2 + \int_{\mathbb{D}} |\phi'(z)f(z)+f'(z)\phi(z)|^2  \frac{1-|z|^2}{V_\mu(z)}\,\operatorname{d}\!A(z) .
\end{align*}
Clearly, $|\phi(0)f(0)|^2 \leq |\phi(0)|^2 \|f\|_{H^2}^2.$ By the triangular inequality, we reduce to the following two summands.

On the one hand,
\begin{align*}
\int_{\mathbb{D}} |f'(z)\phi(z)|^2  \frac{1-|z|^2}{V_\mu(z)}\,\operatorname{d}\!A(z) &\leq \frac{\|b\|_{H^\infty}^2}{L_u}\int_{\mathbb{D}} |f'(z)|^2  (1-|z|^2) \,\operatorname{d}\!A(z) \\
&\lesssim \|f\|_{H^2}^2.
\end{align*}

On the other hand, using the fact that
\[
\phi'(z)= \frac{2u'(z)}{(1-u(z))^2}=\frac{u'(z)}{a(z)(1-u(z))}, \qquad z\in\mathbb{D},
\]
we have that
\begin{align*}
\int_{\mathbb{D}} |\phi'(z)f(z)|^2  \frac{1-|z|^2}{V_\mu(z)}\,\operatorname{d}\!A(z) &=  \int_{\mathbb{D}} \frac{|u'(z)f(z)|^2}{|1-u(z)|^2}  \frac{1-|z|^2}{|a(z)|^2V_\mu(z)}\,\operatorname{d}\!A(z) \\
&\leq \frac{1}{L_u} \int_{\mathbb{D}} \frac{|u'(z)f(z)|^2}{|1-u(z)|^2} (1-|z|^2)\,\operatorname{d}\!A(z).
\end{align*}
As we did in the proof of Lemma \ref{L:phiinEm}, we invoke the operator $T_g$ defined in \eqref{D:Tg}. We have
\[
\frac{u'}{1-u}f = (T_g f)',
\]
for $g=-\log(1-u)$. Again, since $u$ is inner, we have that $g\in \operatorname{BMOA}$, and, therefore,
\begin{align*}
\int_{\mathbb{D}} |\phi'(z)f(z)|^2  \frac{1-|z|^2}{V_\mu(z)}\,\operatorname{d}\!A(z) &\leq \frac{1}{L_u} \| (T_g f)'\|_{A^2_1}^2 \lesssim \| T_g f\|_{H^2}^2 \lesssim \|f\|_{H^2}^2.
\end{align*}
We deduce that
\[
\|M_\phi f\|_{L^2_{a,\mu}}^2\lesssim \|f\|_{H^2}^2.
\]

To conclude the proof, we have to show that $M_\phi f\in E(\mu)$ for every $f\in H^2$. What follows involves some technical details, but it is a fairly standard argument that relies on the density of polynomials in $H^2$ and the fact that $\phi\in E(\mu)$.

We set $\|\phi\|_M=\sup \{\|M_\phi(f)\|_{L^2_{a,\mu}}\colon \|f\|_{H^2}\leq 1\}$. In the first part of proof, we showed that $\|\phi\|_M<\infty$.
We fix $f\in H^2$ and $\varepsilon>0$. Let $P_1$ be a polynomial such that $\|f-P_1\|_{H^2}\leq (2\|\phi\|_M)^{-1}\varepsilon$. Using Lemma \ref{L:phiinEm}, we fix another polynomial $P_2$ such that $\|P_1\phi-P_2\|_{L^2_{a,\mu}}<\varepsilon/2$.  We conclude that
\begin{align*}
\|\phi f-P_2\|_{L^2_{a,\mu}} &\leq  \|\phi f-\phi P_1\|_{L^2_{a,\mu}} +  \|\phi P_1-P_2\|_{L^2_{a,\mu}} \\
&\leq \|\phi\|_M \|f-P_1\|_{H^2} + \|P_1\phi-P_2\|_{L^2_{a,\mu}} <\varepsilon,
\end{align*}
showing that $\phi f\in E(\mu)$ and then $M_\phi$ is bounded from $H^2$ to $E(\mu)$.
\end{proof}

Now, we show that for one-component inner functions, the assumption on the potential condition in Theorem \ref{T:EmbDminHb} follows from an easier condition on the Clark points.

\begin{lemma} \label{L:inf1c}
Let $u$ be a one-component inner function, and let $\mu$ be a measure of the form
\[\mu=\sum_{n}\mu_n\delta_{\zeta_n},\]
where $\{\zeta_n\}_n$ are the Clark atoms of $u$. Then
\[
\inf_{\mathbb{D}}|a|^2 V_\mu>0
\]
if and only if
\[
\mu_n \gtrsim \frac{1}{|u'(\zeta_n)|^2}
\]
uniformly in the atoms $\zeta_n$.
\end{lemma}

\begin{proof}
If $\inf_{\mathbb{D}}|a|^2 V_\mu>0$, the uniform estimation $\mu_n \gtrsim |u'(\zeta_n)|^{-2}$ follows from Lemma \ref{L:Lemma5.1(1)}.

The other implication is more convoluted. We apply Lemma 2.2 in \cite{bessonov2015duality} to the Clark measure $\sigma$. Then there exists $\kappa >0$ such that for every atom $\zeta_n$ the function $u$ is analytic on the open disk
\[
D_{n}(\kappa)=\{z\in\mathbb{C}\colon |z-\zeta_n|<\kappa |u'(\zeta_n)|^{-1}\},
\]
and for every $z\in D_{n}(\kappa)$ we have that
\begin{equation} \label{E:bessonov22}
\frac{|u'(\zeta_n)|}{2} \leq \left|\frac{1-u(z)}{\zeta_n-z}\right|\leq 2|u'(\zeta_n)|.
\end{equation}
Also, denoting $D_\sigma(\kappa)=\bigcup_n D_n(\kappa),$ by Lemma 3.2 in \cite{bessonov2015duality}, there exists $\varepsilon>0$ such that $|1-u(z)|\geq \varepsilon$ for every $z\in\mathbb{D}\setminus D_\sigma(\kappa)$.

Let $C>0$ be a constant such that  \(\mu_n \geq C|u'(\zeta_n)|^{-2}\) uniformly holds. If $z\in D_n(\kappa)$ for some $n$, then by \eqref{E:bessonov22}
\[
|a(z)|^2 V_\mu(z) \geq \frac{|1-u(z)|^2}{4} \frac{C}{|u'(\zeta_n)|^2}\frac{1}{|z-\zeta_n|^2}\geq \frac{C}{16}.
\]
But if $z\in\mathbb{D}\setminus D_\sigma(\kappa)$, then, by \eqref{E:Properties of Vmu},
\[
|a(z)|^2 V_\mu(z) \geq \frac{\varepsilon^2}{4} \frac{\mu(\mathbb{T})}{4}.
\]
We conclude that
\[
\inf_{\mathbb{D}} |a|^2V_\mu \geq \min\left\{\frac{C}{16},\frac{\varepsilon^2\mu(\mathbb{T})}{16}\right\}>0.
\]
\end{proof}

We are now in a position to sum up the proof of Theorem \ref{T:main1}. In Section \ref{S:proofT1measure} we have already established the necessity of \eqref{E:main1expressionmu}, \eqref{E:main1atoms} and \eqref{E:main1potential}. Conversely, if $\mu$ satisfies the conditions \eqref{E:main1expressionmu}-\eqref{E:main1potential}, then Theorem \ref{T:mainTemb} gives the embedding $\mathcal{H}(b)\hookrightarrow\mathcal{D}_\mu$. Also, in view of Lemma \ref{L:inf1c} we obtain the potential condition in Theorem \ref{T:EmbDminHb}, so that $\mathcal{D}_\mu$ embeds in $\mathcal{H}(b)$ and thus, finally, $\mathcal{H}(b)=\mathcal{D}_\mu$.

\section{Examples} \label{S:examples}

In this section, which is divided into two subsections, we provide examples to clarify that the conditions of Theorem \ref{T:main1} are not trivial, more precisely, we provide a one-component inner function satisfying conditions \eqref{E:main1expressionmu}-\eqref{E:main1potential}, and a  class of one-component inner functions for which \eqref{E:main1potential} does not hold.

We assume throughout this section that $\mu$ is the explicit discrete measure
\begin{equation}\label{E:muexamples}
\mu=\sum_{n} \frac{1}{|u'(\zeta_n)|^2}\delta_{\zeta_n}.
\end{equation}
With this assumption, Lemma \ref{L:Lemmaequiv} can be reformulated as follows:
\[
\sup_{\xi \in \mathbb{T}\setminus \rho(u)} \limsup_{z\to\xi}|1-u(z)|^2V_\mu(z) <\infty
\]
if and only if
\[
\sup_m \sum_{n\neq m} \frac{1}{|u'(\zeta_n)|^2}\frac{1}{|\zeta_n-\zeta_m|^2}<\infty.
\]

First, we establish a useful result that applies to every one-component inner function.

\begin{lemma} \label{L:supVm}
Let $u$ be a one-component inner function, and let $\mu$ be as in \eqref{E:muexamples}. Then $\sup_{\mathbb{D}} |1-u|^2V_\mu <\infty$ if and only if
\begin{enumerate}[(i)]
\item \label{L:supVm(i)} $\sup_{\rho(u)} V_\mu<\infty$, and
\item \label{L:supVm(ii)} $\sup_{\xi\in\mathbb{T}\setminus\rho(u)} \limsup_{z\to\xi}|1-u(z)|^2V_\mu(z) <\infty$.
\end{enumerate}
\end{lemma}

\begin{proof}
To start, we assume that $\sup_{\mathbb{D}} |1-u|^2V_\mu <\infty$. Then, by Theorem \ref{T:main1}, we know that $\mathcal{H}(b)=\mathcal{D}_\mu$, for $b=(1+u)/2$. In particular, $K_u\hookrightarrow\mathcal{D}_\mu$ and, as shown in the proof of Theorem 1.4 of \cite{bellavita2023embedding}, $\sup_{\rho(u)} V_\mu<\infty.$ The other condition follows from the fact that, by assumption, $|1-u|^2 V_\mu$ is bounded on $\mathbb{D}$ and extends continuously to every point in $\mathbb{T}\setminus\rho(u)$, by Lemma \ref{L:Lemma5.1(1)}.

We prove the converse implication. Let $M$ be a majorant of both suprema in \eqref{L:supVm(i)} and \eqref{L:supVm(ii)}. For every integer $N\geq 1$, we apply the maximum principle to the subharmonic function given by the partial sum
\[
 \Phi_N(z)=
\sum_{n=1}^N \frac{1}{|u'(\zeta_n)|^2}\left|\frac{1-u(z)}{z-\zeta_n}\right|^2.
\]
In particular, it is enough to show that for every $N\geq 1$
\[\sup_{\xi\in\mathbb{T}} \limsup_{z\to\xi} \Phi_N(z) \leq 4M.\]
The function $\Phi_N$ is continuous on $\overline{\mathbb{D}}\setminus\rho(u)$, by Lemma \ref{L:Lemma5.1(1)}. For $\xi\in\mathbb{T}\setminus\rho(u)$, we have
\[
 \Phi_N(\xi)=\limsup_{z\to \xi}\Phi_N(z)\le \limsup_{z\to\xi}|1-u(z)|^2V_\mu(z) \le M.
\]
When $\xi\in \rho(u)$,  
\begin{align*}
 \limsup_{z\to\xi}\Phi_N(z)&\le 4 \limsup_{z\to \xi}
 \sum_{n=1}^N \frac{1}{|u'(\zeta_n)|^2}\left|\frac{1}{z-\zeta_n}\right|^2\\
 & =4\sum_{n=1}^N \frac{1}{|u'(\zeta_n)|^2}\frac{1}{|\xi-\zeta_n|^2}\\
&\leq 4V_{\mu}(\xi)\le 4M.
\end{align*}

By subharmonicity, $\sup_{z\in\mathbb{D}}\Phi_N(z)\le 4M$, and then for every $z\in \mathbb{D}$,
\[
 |1-u(z)|^2V_{\mu}(z)=\lim_{N\to+\infty}\Phi_N(z)\le 4M.
\]
\end{proof}

\subsection{An explicit example of a one-component inner function satisfying \eqref{E:main1potential}}\label{Subsection-example}
We prove that the singular inner function $u$ associated to $\delta_1$,
\[u(z)=\exp{\left(\frac{z+1}{z-1}\right)}, \qquad z\in\mathbb{D},\]
and the measure
\[\mu=\sum_n \frac{1}{|u'(\zeta_n)|^2}\delta_{\zeta_n},\]
where $\zeta_n$ are the Clark atoms of $u$, satisfy the potential condition $\sup_\mathbb{D}|1-u|^2V_\mu <\infty$. This provides an explicit example of the equality $\mathcal{H}(b)=\mathcal{D}_\mu$, where $\mu$ is a measure with infinitely many atoms.

Notice that $u$ is a one-component inner function with $\rho(u)=\{1\}$ and its first derivative satisfies
\[
|u'(\zeta)|=\frac{2}{|\zeta-1|^2},\qquad \zeta\in\mathbb{T}\setminus\{1\}.
\]
The atoms of the Clark measure are anchored at
\[
\{\zeta_n\}_{n\in\mathbb{Z}} = \left\{\frac{2n\pi i+1}{2n\pi i-1}\colon n\in\mathbb{Z}\right\}.
\]
In particular, since
\[
\zeta_n-1=\frac{2}{2n\pi i-1}, \qquad n \in \mathbb{Z},
\]
we can compute
\[
\frac{1}{|u'(\zeta_n)|}=\frac{|\zeta_n-1|^2}{2}=\frac{2}{4n^2\pi^2+1}\asymp \frac{1}{n^2}, \qquad \mbox{as } n\to\infty,
\]
and
\[
\zeta_k-\zeta_j = \frac{4\pi i(j-k)}{(2k\pi i-1)(2j\pi i -1)}, \qquad k\neq j.
\]

 \begin{theorem}\label{T:mainexample}
 Let $u$ be the one-component inner function
 \[u(z)=\exp{\left(\frac{z+1}{z-1}\right)}, \qquad z\in\mathbb{D}.\]
 If $b=(1+u)/2$ and $\mu$ is as in \eqref{E:main1expressionmu}, then we have the equality
 \[
 \mathcal{H}(b)=\mathcal{D}_\mu.
 \]
 \end{theorem}

\begin{proof}
We use Lemma \ref{L:supVm} to show that the condition of Theorem \ref{T:main1} is satisfied. We first notice that
\[
V_\mu(1)=\sum_{n} \frac{1}{|u'(\zeta_n)|^2}\frac{1}{|\zeta_n-1|^2}\asymp\sum_{n\geq 1} \frac{1}{n^4} n^2<\infty.
\]
To deal with $\sup_{\mathbb{T}\setminus\rho(u)} |a|^2V_\mu$, we appeal to Lemma \ref{L:Lemmaequiv}. We have that
\begin{align*}
\sum_{n\neq m}\frac{1}{|u'(\zeta_n)|^2} \frac{1}{|\zeta_n-\zeta_m|^2} \asymp  \sum_{n\neq m} \frac{1}{n^2}\frac{m^2}{(n-m)^2},
\end{align*}
and splitting the sum in three parts, for $n>m,$ $m/2<n<m$ and $n\leq m/2$, we see that
\[
\sup_m \sum_{n\neq m} \frac{1}{n^2}\frac{m^2}{(n-m)^2} <\infty,
\]
and we conclude that $\mathcal{H}(b)=\mathcal{D}_\mu$.
\end{proof}

For this particular inner function $u$, we are able to prove that  the Cauchy transform $\mathcal{C}_\sigma$ is bounded without invoking Theorem \ref{T:main2}. The discrete Hilbert transform on $\ell^2(\mathbb{Z})$ naturally appears in this case. For other discussions about the relationship  between the discrete Hilbert transform and the model spaces we refer to \cite{BELLAVITA2025110708,Eoff1995TheDN}.

\begin{proposition}\label{L:excauchy}
Let $u$ be a one-component inner function defined as in Theorem \ref{T:mainexample}. Then the Cauchy transform $\mathcal{C}_\sigma$ is bounded in $L^2(\sigma)$.
\end{proposition}

\begin{proof}
The Cauchy transform $\mathcal{C}_\sigma\colon L^2(\sigma)\to L^2(\sigma)$ defined as
\[
\mathcal{C}_\sigma f(\zeta_n)=\sum_{m\neq n}\frac{f(\zeta_m)}{|u'(\zeta_m)|} \frac{1}{1-\overline{\zeta_m}\zeta_n},
\]
in this setting becomes
\[
\mathcal{C}_\sigma f(\zeta_n)=\frac{2n\pi i-1}{4\pi i}\sum_{m\neq n}\frac{f(\zeta_m)}{|u'(\zeta_m)|} \frac{2m\pi i+1}{n-m}.
\]
We define the sequence $x=(x_m)_{m\in\mathbb{Z}}$ as
\[
x_m = (2m\pi i +1)\frac{f(\zeta_m)}{|u'(\zeta_m)|}, \qquad m\in\mathbb{Z},
\]
and we notice that
\begin{align*}
\sum_{m\in\mathbb{Z}} |x_m|^2 \asymp \sum_{m\in\mathbb{Z}} \frac{|f(\zeta_m)|^2}{|u'(\zeta_m)|} = \|f\|_{L^2(\sigma)}^2 <\infty.
\end{align*}

The discrete Hilbert transform is
\[
(\mathcal{H}_d x)(n)=\sum_{m\neq n}\frac{x(m)}{m-n}, \qquad n\in\mathbb{Z}.
\]
Plancherel and Polya \cite{Plancherel1936FonctionsEE} proved that $\mathcal{H}_d$ is bounded on $\ell^2(\mathbb{Z})$ and, consequently, we have that
\begin{align*}
\|\mathcal{C}_\sigma f\|_{L^2(\sigma)}^2 &= \sum_n \frac{1}{|u'(\zeta_m)|} \left|\frac{2n\pi i-1}{4\pi i}\sum_{m\neq n}\frac{f(\zeta_m)}{|u'(\zeta_m)|} \frac{2m\pi i+1}{n-m}\right|^2 \\
&\asymp \sum_n |\mathcal{H}_d x(n)|^2 \lesssim \|x\|_{\ell^2(\mathbb{Z})} \asymp \|f\|_{L^2(\sigma)}^2.
\end{align*}
\end{proof}

\subsection{Optimality of condition \eqref{E:main1potential}}\label{Subsection-counterexample}

Here we show that the potential condition \eqref{E:main1potential} is not induced by the fact that $u$ is one-component. More precisely, we exhibit a class of one-component Blaschke products for which \eqref{E:main1potential} is not satisfied. We introduce the Cayley transform
\[
\phi\colon\mathbb{C}_+\to\mathbb{D},\qquad \phi(w)=\frac{w-i}{w+i}.
\]
Also, $\rho$ denotes the pseudohyperbolic distance in the upper half-plane,
\[\rho(u,w)=\bigg|\frac{u-w}{u-\overline{w}}\bigg|,\qquad u,w\in\mathbb{C}_+.\]

In what follows we need interpolating sequences for the Hardy space $H^2(\mathbb{C}_+)$ of the upper half-plane. Since there will be no confusion possible, we will simply call these sequences interpolating. Recall that a sequence of points $\Lambda=(\lambda_n)$ in the upper half-plane is interpolating if for every sequence $(v_n)$ with $\sum \Im(\lambda_n)|v_n|^2<\infty$ there exists $f\in H^2(\mathbb{C}_+)$ such that $f(\lambda_n)=v_n$, $n \geq 1$. Starting from Carleson’s characterization of interpolating sequences in $H^{\infty}(\mathbb{D})$, Shapiro-Shields characterized interpolating sequences for $H^2(\mathbb{D})$. Translated to the upper half-plane, $\Lambda$ is interpolating if and only if it is separated in
the metric $\rho$ and $\mu_{\Lambda}= \sum \Im \lambda_n \delta_{\lambda_n}$ is a Carleson measure. Note that a measure $\mu$ supported in the upper half-plane is Carleson if and only if $\sum_{\lambda_n\in S_I} \Im \lambda \lesssim |I|$, where $I\subset \mathbb{R}$ is an interval and
$S_I=\{x+iy:x\in I,0<y<|I|\}$ is a Carleson box. We recall that an interpolating sequence is automatically a Blaschke sequence, i.e., satisfies $\sum_n \frac{\Im\lambda_n}{1+|\lambda_n|^2}<\infty$. With these elements in mind, we can state the following result.

\begin{proposition}
Let $\Lambda=\{\lambda_n\}_{n\geq 1}$ be an interpolating sequence in the upper half-plane $\mathbb{C}_+$, enumerated so that $\Re(\lambda_n)\leq\Re(\lambda_{n+1})$ and such that $|\lambda_n|\to\infty$ as $n\to\infty$. We also assume that $\Lambda$ satisfies the following two properties.
\begin{enumerate}[(i)]
\item There exists $c\in (0,1)$ such that for every $n\geq 1$ \begin{equation}\label{E:sep}
 \rho(\lambda_n,\lambda_{n+1})\le c.
\end{equation}
\item $\Re(\lambda_n)\geq 0$ for every $n\geq 1$.
\end{enumerate}
Then the Blaschke product on the disk
\[
B_{\phi(\Lambda)}(z)= \prod_{n\geq 1}\frac{\overline{\phi(\lambda_n)}}{|\phi(\lambda_n)|}\frac{\phi(\lambda_n)-z}{1-\overline{\phi(\lambda_n)}z},\qquad z\in\mathbb{D},
\]
is a one-component inner function that does not satisfy the potential condition \eqref{E:main1potential}.
\end{proposition}

\begin{proof}
We introduce $B_\Lambda$, the Blaschke product on $\mathbb{C}_+$ associated to $\Lambda$. Since $\Lambda$ is interpolating, so will be $\phi(\Lambda).$ In view of \eqref{E:sep} we can apply \cite[Theorem 6]{cimamortini}, showing that $B_{\phi(\Lambda)}$ is one-component. Since such a property is conserved by the conformal mapping $\phi$, $B_\Lambda$ is one-component, as well. Now, we show that $B_{\phi(\Lambda)}$ does not satisfy the potential condition \eqref{E:main1potential}. Since $\phi(\lambda_n)\to 1$ as $n\to\infty$, the boundary spectrum of $B_{\phi(\Lambda)}$ satisfies $\rho(B_{\phi(\Lambda)})=\{1\}$. We focus on the (infinitely many) Clark atoms $e^{it_n}, n>0,$ with $t_n\in(0,\pi/2)$. We order them so that $(t_n)_n$ is monotonically decreasing to $0$ as $n\to+\infty$. The function $\phi^{-1}(e^{it})$ for $t\in(0,2\pi)$ increases as $t$ increases, and maps conformally $(0,\pi/2)$ to $(-\infty,-1)$. We move to the Clark atoms $x_n=\phi^{-1}(e^{it_n})$ of $B_\Lambda$ on $\mathbb{R}_-$. With this enumeration, $x_{1}$ is the closest to $-1$ and $-\infty<x_{n+1}< x_n< -1$, for every $n>0$. Using \cite[Equation 26]{BARANOV2005116}, we have that
\[
\frac{1}{|B_{\Lambda}'(x)|}\asymp \operatorname{dist}(x,\Lambda)\asymp |x|,\qquad x<0.
\]

Let us move back to the disk. Since
\[
e^{it_n}=\frac{x_n-i}{x_n+i},
\]
we have that uniformly in $n>0$
\begin{equation} \label{E:exampleneg1}
t_n=\arg \frac{x_n-i}{x_n+i}=\arctan \frac{-2x_n}{x_n^2-1} \asymp \frac{1}{|x_n|}.
\end{equation}
 Also,
\[
 |B'_{\phi(\Lambda)}(z)|=\frac{2}{|1-z|^2}\left|B_{\Lambda}'\left(i\frac{1+z}{1-z}\right)\right|,\qquad z\in\overline{\mathbb{D}}\setminus\{1\},
\]
and hence
\[
 |B'_{\phi(\Lambda)}(e^{it_n})|\asymp \frac{1}{t_n^2}|B'_{\Lambda}(x_n)|\asymp \frac{1}{t_n^2}\frac{1}{|x_n|},
\]
 uniformly in $n>0$. To show that the potential condition is not satisfied it is enough to consider the Clark points $t_n\in[0,\pi/2]$. For this, we use Lemmas \ref{L:Lemmaequiv} and \ref{L:supVm}. For $m\geq1$,
 \begin{align*}
 \sum_{n\neq m}\frac{1}{|B'_{\phi(\Lambda)}(e^{it_n})|^2}\frac{1}{|e^{it_n}-e^{it_m}|^2} \gtrsim \sum_{n<m}\frac{t_n^4|x_n|^2}{(t_n-t_m)^2}  \geq\sum_{n<m}\frac{t_n^4|x_n|^2}{t_n^2} \asymp m,
 \end{align*}
 where the second inequality follows from the fact that $t_n>t_m>0$, and the conclusion follows from \eqref{E:exampleneg1}. In particular,
 \[
 \sup_{m\in\mathbb{N}}\sum_{n\neq m}\frac{1}{|B'_{\phi(\Lambda)}(e^{it_n})|^2}\frac{1}{|e^{it_n}-e^{it_m}|^2}=+\infty,
 \]
 proving that $B_{\phi(\Lambda)}$ does not satisfy the potential condition \eqref{E:main1potential}.
\end{proof}

For explicit examples of sequences $\Lambda$, one can consider $\lambda_n=n+i$ or, more generally, \[
 \lambda_n=n^{\alpha}+i\frac{1}{n^{1-\alpha}},
\]
for $\alpha\in (0,1]$.

A direct computation shows that $\Lambda$ is a Blaschke sequence and that $0<c_1<\rho(\lambda_n,\lambda_{n+1})<c_2<1$ for every $n$. It is also straightforward to check that $\mu=\sum_n \Im \lambda_n \delta_{\lambda_n}$ satisfies the Carleson condition.

We remark that, if we symmetrize this set $\Lambda$, then we get positive examples for the potential condition \eqref{E:main1potential}: setting
\[\widetilde{\Lambda}=\{x+iy\in\mathbb{C}\colon |x|+iy\in\Lambda\},\]
one can show that the Blaschke product $B_{\phi(\widetilde{\Lambda})}$ on $\mathbb{D}$ is a one-component inner function that satisfies the potential condition \eqref{E:main1potential}, for certain values of $\alpha$.

\section{A Perturbation result: proof of Theorem \ref{T:perturbation}} \label{S:perp-result}
\begin{proof}[Proof of Theorem \ref{T:perturbation}]
We verify that $\theta$ is one-component by verifying the conditions of Theorem \ref{T:Bessonov}. Since $\lim_n\sigma_n\alpha_n =0$, it is easy to see that $\{\zeta_n\}_n$ and $\{t_n\}_n$ have the same accumulation points. In particular, $\tau(\lambda)$ coincides with $\tau(\sigma)$. Moreover, by the definition of $\alpha$,
\begin{align*}
|t_n-\zeta_n| \leq  \sigma_n\alpha_n
&\leq B_{\sigma}\min\left(|\zeta_n-\zeta_n^{+}|, |\zeta_n-\zeta_n^{-}|\right)\|\alpha\|_{\infty} \\
&\leq \frac{A_\sigma}{3B_\sigma}\min\left(|\zeta_n-\zeta_n^{+}|, |\zeta_n-\zeta_n^{-}|\right) \\
&\leq \frac{1}{3} \min\left(|\zeta_n-\zeta_n^{+}|, |\zeta_n-\zeta_n^{-}|\right),
\end{align*}
so that if $\zeta_n^+=\zeta_{n+1}$ and $\zeta_n^-=\zeta_{n-1}$, then $t_n$ has $t_{n+1}$ and $t_{n-1}$ as neighbors, and the first condition of Theorem \ref{T:Bessonov} holds true.

Since $|\epsilon_n|\leq \|\alpha\|_\infty \sigma_n\leq\sigma_n/2$, we have that $\lambda_n \asymp \sigma_n$. Also, if $s=n+1$ or $s=n-1$,
\[
\sigma_s\alpha_s\leq \frac{A_\sigma}{3B_\sigma}|\zeta_s-\zeta_n|\leq \frac{1}{3}  \min\left(|\zeta_n-\zeta_n^{+}|, |\zeta_n-\zeta_n^{-}|\right),
\]
by Theorem \ref{T:Bessonov}, and trivially,
\[
\sigma_s\alpha_s\leq \frac{A_\sigma}{3B_\sigma}|\zeta_s-\zeta_n|\leq \frac{A_\sigma}{3B_\sigma}\max\left(|\zeta_n-\zeta_n^{+}|, |\zeta_n-\zeta_n^{-}|\right).
\]
Considering two consecutive atoms $t_n, t_{s}$ of $\lambda$, we have that
\begin{align*}
|t_n-t_s|&\leq |t_n-\zeta_{n}|+ |\zeta_{n}-\zeta_s|+ |\zeta_s-t_s| \\
&\leq \bigg(\frac{A_\sigma}{3B_\sigma}+ \frac{B_\sigma}{A_\sigma}+\frac{1}{3}\bigg)\min\left(|\zeta_n-\zeta_n^{+}|, |\zeta_n-\zeta_n^{-}|\right).
\end{align*}
On the other hand,
\begin{align*}
|t_n-t_{s}| &\geq |\zeta_n-\zeta_s|-|t_n-\zeta_n|-|t_s-\zeta_s|\\
&\geq\frac{A_\sigma}{B_\sigma}\max\left(|\zeta_n-\zeta_n^{+}|, |\zeta_n-\zeta_n^{-}|\right)-\sigma_{n}\alpha_{n}-\sigma_{s}\alpha_{s}\\
&\geq \frac{A_\sigma}{3B_\sigma}\max\left(|\zeta_n-\zeta_n^{+}|, |\zeta_n-\zeta_n^{-}|\right).
\end{align*}
Since $\lambda_n\asymp\sigma_n$ and $\sigma_n,\zeta_n$ satisfy the second condition of Theorem \ref{T:Bessonov}, so will $\lambda_n$ and $t_n$.

Finally, we show that the third condition is verified, i.e., that
\[
\mathcal{C}_\lambda 1(z)=\int_{\mathbb{T}\setminus\{z\}} \frac{\operatorname{d}\!\lambda(\xi)}{1-\overline{\xi}z}, \qquad z\in\mathbb{T},
\]
is uniformly bounded on $z\in a(\lambda)$. We write
\begin{eqnarray}
|\mathcal{C}_\lambda 1(t_n)| &=& \bigg|\sum_{m\neq n}\frac{\lambda_m}{t_m-t_n} \bigg| \notag\\
&\leq& \bigg|\sum_{m\neq n}\frac{\lambda_m}{\zeta_m-\zeta_n} \bigg| + \bigg|\sum_{m\neq n}\bigg(\frac{\lambda_m}{t_m-t_n} -\frac{\lambda_m}{\zeta_m-\zeta_n}\bigg) \bigg|. \label{E:examples}
\end{eqnarray}
Concerning the first summand,
\begin{align*}
\bigg|\sum_{m\neq n}\frac{\lambda_m}{\zeta_m-\zeta_n} \bigg| &\leq \bigg|\sum_{m\neq n}\frac{\sigma_m}{\zeta_m-\zeta_n} \bigg| + \bigg|\sum_{m\neq n}\frac{\epsilon_m}{\zeta_m-\zeta_n} \bigg|\\
&\leq |\mathcal{C}_\sigma1(\zeta_n)| +\sum_{m\neq n}\frac{\sigma_m\alpha_m}{|\zeta_m-\zeta_n|}.
\end{align*}
We notice that $|\mathcal{C}_\sigma(1)|$ is uniformly bounded on $a(\sigma)$ since $u$ is one-component, and that the other term is finite by condition \eqref{E:pertubation2}. Concerning the second summand in \eqref{E:examples}, we have
\begin{align*}
\bigg|\sum_{m\neq n}&\frac{\lambda_m}{t_m-t_n} -\frac{\lambda_m}{\zeta_m-\zeta_n} \bigg|\leq \sum_{m\neq n} \frac{\lambda_m\big(|\zeta_m-t_m|+|\zeta_n-t_n|\big)}{|t_m-t_n|\,|\zeta_m-\zeta_n|}\\
&\leq |t_n-\zeta_n|\sum_{m\neq n} \frac{\lambda_m}{|t_m-t_n|\,|\zeta_m-\zeta_n|}+\sum_{m\neq n} \frac{\lambda_m|\zeta_m-t_m|}{|t_m-t_n|\,|\zeta_m-\zeta_n|}.
\end{align*}
Notice that, for fixed $n$, for every $m\neq n$,
\begin{align*}
|\zeta_m-\zeta_n| &\leq |\zeta_m-t_m|+|t_m-t_n|+|t_n-\zeta_n|\\
&\leq \alpha_m\sigma_m + \alpha_n\sigma_n+|t_m-t_n|\\
&\lesssim \lambda_m+\lambda_n +|t_m-t_n|\\
&\leq B_\lambda(|t_m-t_m^\pm|+|t_n-t_n^\pm|)+|t_m-t_n|\\
&\leq (2 B_\lambda+1)|t_m-t_n|,
\end{align*}
where $B_\lambda$ is the constant that realizes the second condition of Theorem \ref{T:Bessonov} for the measure $\lambda$, that we have shown to hold. In a similar way one shows that
\[
|t_m-t_n|\lesssim |\zeta_m-\zeta_n|,
\]
uniformly for $m\neq n$.
By Lemma \ref{L:integralestimate} and Theorem \ref{T:Bessonov}, we have that
\begin{eqnarray*}
\sum_{m\neq n} \frac{\lambda_m}{|t_m-t_n|\,|\zeta_m-\zeta_n|}
&\lesssim& \sum_{m\neq n} \frac{\sigma_m}{|\zeta_m-\zeta_n|^2}\\
&=& \int_{\mathbb{T}\setminus I_{\zeta_n}}\frac{1}{|\zeta-\zeta_n|^2}\operatorname{d}\!\sigma(\zeta)\lesssim\frac{1}{\sigma_n},
\end{eqnarray*}
where $I_\zeta$ is the open arc having for extremes $\zeta_n^+,\zeta_n^-$. In particular,
\[
|t_n-\zeta_n|\sum_{m\neq n} \frac{\lambda_m}{|t_m-t_n|\,|\zeta_m-\zeta_n|} \lesssim \alpha_n
\]
is uniformly bounded in $n$. Concerning the remaining term,
\begin{eqnarray*}
\sum_{m\neq n} \frac{\lambda_m|\zeta_m-t_m|}{|t_m-t_n|\,|\zeta_m-\zeta_n|}
&\lesssim& \sum_{m\neq n} \frac{\lambda_m\sigma_m\alpha_m}{|\zeta_m-\zeta_n|^2}\\
&\leq& \sup_{m\neq n}\frac{\sigma_m}{|\zeta_m-\zeta_n|}\sum_{m\neq n}\frac{\lambda_m\alpha_m}{|\zeta_m-\zeta_n|}.
\end{eqnarray*}
Now, since $\sigma_m\leq B_\sigma |\zeta_m-\zeta_m^\pm|\leq B_\sigma|\zeta_m-\zeta_n|$ for every $m\neq n$, we conclude using again the assumption \eqref{E:pertubation2} that $\mathcal{C}_\lambda1$ is uniformly bounded on $a(\lambda)$. Therefore, by Theorem \ref{T:Bessonov}, we conclude that $\theta$ is a one-component inner function.
\end{proof}

It was observed at the end of the proof that $\sigma_m\leq B_\sigma |\zeta_m-\zeta_m^\pm|\leq B_\sigma|\zeta_m-\zeta_n|$ for every $m\neq n$. In particular, $\alpha\in\ell^1$ implies condition \eqref{E:pertubation2} of \ref{T:perturbation}. As an application of Theorem \ref{T:perturbation}, we are able to construct numerous examples of equalities $\mathcal{H}(b)=\mathcal{D}_\mu$, starting from the explicit one that was established in Subsection \ref{Subsection-example}.

\begin{corollary}\label{P:examples}
Let $u$ be a one-component inner function, let $\{\zeta_n\}_n$ be its Clark atoms and $\sigma$ its Clark measure. We also assume that $\mathcal{H}(b)=\mathcal{D}_\nu$, where $b=(1+u)/2$ and $\nu$ is a suitable measure satisfying the conditions of Theorem \ref{T:main1}. We consider a sequence of positive numbers $\alpha=(\alpha_n)_n\in \ell^2$ such that $\|\alpha\|_{\infty} \leq\min\{(3B_\sigma)^{-1},A_\sigma/3B_\sigma^2, 1/2\}$. If we take points $t_n\in\mathbb{T}$ and positive numbers $\lambda_n=\sigma_n + \epsilon_n$ that satisfy the assumptions of Theorem \ref{T:perturbation}, then we have that the measure
\[
\lambda=\sum_{n}\lambda_n \delta_{t_n}
\]
is the Clark measure associated to a one-component inner function $\theta$, that satisfies $\mathcal{H}( (1+\theta)/2)=\mathcal{D}_\mu$, with
\[
\mu=\sum_{n}\lambda_n^2 \delta_{t_n}.
\]
\end{corollary}

\begin{proof}
We show that the sequence $\alpha$ satisfies the condition \eqref{E:pertubation2} of Theorem~\ref{T:perturbation}. For every $n$, by the Cauchy-Schwarz inequality
\[
\sum_{m\neq n}\frac{\sigma_m\alpha_m}{|\zeta_n-\zeta_m|}\leq \|\alpha\|_2\bigg(\sum_{m\neq n}\frac{\sigma_m^2}{|\zeta_n-\zeta_m|^2}\bigg)^\frac{1}{2},
\]
and this quantity is uniformly bounded by Theorem \ref{T:main1} and Lemma \ref{L:Lemmaequiv}. Then, $\theta$ is a one-component inner function. We already know that $\rho(\theta)=\rho(u)$. It is also clear by the proof of Theorem \ref{T:perturbation} that $|\eta-\zeta_n|\asymp |\eta-t_n|$ for every $\eta\in\rho(\theta)$, uniformly in $n$. In particular,
\[
V_\mu(\eta)\asymp V_\nu(\eta), \qquad \eta\in\rho(u).
\]
Again, reasoning as in the proof of Theorem \ref{T:perturbation},
\[
\sup_m\sum_{n\neq m} \frac{\sigma_n^2}{|\zeta_n-\zeta_m|^2}\asymp \sup_m\sum_{n\neq m} \frac{\lambda_n^2}{|t_n-t_m|^2},
\]
and, by Theorem \ref{T:main1} and Lemma \ref{L:Lemmaequiv}, we have that
\[
\sup_{\eta \in \mathbb{T}\setminus \rho(\theta)} |1-\theta(\eta)|^2V_\mu(\eta)<\infty.
\]
As a result, using  Lemma \ref{L:supVm} we conclude that the inner function $\theta$ satisfies the condition of Theorem \ref{T:main1}, consequently,
$\mathcal{H}( (1+\theta)/2)=\mathcal{D}_\mu$.
\end{proof}

\section{A variation of the Brown-Shields conjecture} \label{S:brownshields}
The Brown-Shields conjecture concerns cyclicity in Dirichlet spaces.  Given a space $X$ that is closed under the action of the forward shift, a function $f\in X$ is cyclic if the set of products $pf$ with polynomials $p$ is dense in $X$.  Specifically, the conjecture, posed by Brown and Shields \cite{Brown1984CyclicVI}, asserts that a function $f$ in the classical Dirichlet space is cyclic if and only if it is an outer function and its zero set on the unit circle $\mathcal{Z}(f)$ has zero capacity. Here,
\[\mathcal{Z}(f) = \lbrace \zeta \in \mathbb{T} : \lim_{r \to 1^-} |f(r\zeta)| = 0 \rbrace.\]

As it turns out, a positive answer to this conjecture can be given in certain $\mathcal{H}(b)$ spaces when they are actually equal to a $\mathcal{D}_\mu$ space. While we do not really use techniques from the preceding sections, it seemed interesting to mention the following result which explores the equality of spaces discussed previously. In particular, the fact that this equality holds for more general symbols than rational ones may give an additional motivation for this characterization.

\begin{theorem}\label{T:main3}
Let $b$ be an analytic function with $\|b\|_{H^\infty}=1$ and  $\mathcal{H}(b)=\mathcal{D}_\mu$. If $f \in \mathcal{H}(b)$ is such that the set
\begin{equation}\label{E:Tmain3set}
\operatorname{supp}(\mu)\cap \{z\in\overline{\mathbb{D}}\colon \liminf_{w\to z}|f(w)|=0\}
\end{equation}
 is countable, then the following assertions are equivalent:
\begin{itemize}
\item[(i)] $f$ is cyclic in $\mathcal{H}(b)$.
\item[(ii)] $f$ is an outer function and the $\mathcal{H}(b)$-capacity of $\mathcal{Z}(f)$ is zero.
\end{itemize}
\end{theorem}

We point out that when the measure $\mu$ takes the specific form given in \eqref{E:main1expressionmu}, then the set in \eqref{E:Tmain3set} is automatically countable. This provides an explicit characterization for the cyclic vectors of such $\mathcal{H}(b)$ spaces.

Several recent works have advanced the study of cyclic vectors in de Branges--Rovnyak spaces. Fricain, Mashreghi, and Seco \cite{Seco} characterized the cyclic vectors in the $\mathcal{H}(b)$ spaces for which the Toeplitz operator $T_{a/\overline{a}}$, where $a$ is the Pythagorean mate of $b$, is invertible in $L^2(\mathbb{T})$. Bergman \cite{Bergman2023OnCI} characterized the cyclic elements of $\mathcal{H}(b)$ by relating the boundary spectrum of their outer parts to $\rho(b)$. Furthermore, Fricain and Lebreton \cite{Fricain2024CyclicityOT} provided sufficient conditions for the cyclicity of $f\in \mathcal{H}(b)$ using the Corona Theorem.

In this section we consider another approach introduced by Fricain and Grivaux in \cite{FricainGrivaux}. For a set $E\subset \mathbb{T}$, we define the quantities
\[
c_{b,1}(E)=\inf\{\|f\|^2_{\mathcal{H}(b)} \colon f \in \mathcal{H}(b),  |f|\geq 1 \text{ a.e. on a neighbourhood of } E\}
\]
and
\[
c_{b,2}(E)=\inf\{\|f\|^2_{\mathcal{H}(b)} \colon f \in  \mathcal{H}(b),  |f|=1 \text{ a.e. on a neighborhood of } E\}.
\]
It is clear that
\[
c_{b,1}(E)\leq c_{b,2}(E),
\]
but in general it is not known whether the reverse holds. However, when $\mathcal{H}(b)=\mathcal{D}_\mu$, we can show that $c_{b,1}$ and $c_{b,2}$ are comparable. This provides a partial answer to Question 3.4 posed by Fricain and Grivaux in \cite{FricainGrivaux}, concerning whether $c_{b,1}$ and $c_{b,2}$ are comparable in every $\mathcal{H}(b)$ space.

\begin{proposition}\label{P:cyclicity}
Let $b \in H^{\infty}$ with $\|b\|_{H^\infty} = 1$ and $\mu$ a finite measure on $\mathbb{T}$. If
$\mathcal{H}(b) = \mathcal{D}_{\mu}$, then there exists a positive constant $A$ such that for every $E \subset \mathbb{T}$
\[
c_{b,2}(E) \le A c_{b,1}(E).
\]
\end{proposition}
\begin{proof}
Let $E\subset \mathbb{T}$. Just in this proof and to save space, we write a.e.n. $E$ for ‘‘almost everywhere in a neighborhood of the set $E$".
We first note that
\begin{align}
\nonumber c_{b,2}(E) &= \inf\{\|f\|^2_{\mathcal{H}(b)} : f \in \mathcal{H}(b), |f| = 1 \text{ a.e.n. } E\}\\
 &\lesssim\inf\{\|f\|^2_{\mathcal{D}_{\mu}} : f \in \mathcal{H}(b),|f| = 1 \text{ a.e.n. }E\}.
 \label{Expr1}
\end{align}
Using successively Corollary 7.6.2 and Theorem 7.5.2 in \cite{primer-fkmr}, we have
\[
\|f\|^2_{\mathcal{D}_\mu} \geq  \|f_o\|^2_{\mathcal{D}_\mu} \geq  \|f_o\wedge 1\|^2_{\mathcal{D}_\mu},
\]
where $f_o$ is the outer part of $f$ and $f_o\wedge 1$ is the outer function with boundary values
\[
|(f_o \wedge 1)(e^{it})|=\min(|f(e^{it})|,1).
\]
From the above inequalities it follows that in both infima $c_{b,1}$ and $c_{b,2}$
we can replace $f$ by its outer part, which conserves the condition on the neighborhood of $E$.
 We now get
\begin{align*}
c_{b,1}(E) &\asymp  \inf\{\|f\|^2_{\mathcal{D}_\mu} : f \in \mathcal{H}(b),|f| \ge 1 \text{ a.e.n. }E\}\\
&=\inf\{\|f_o\|^2_{\mathcal{D}_\mu} : f_o \in \mathcal{H}(b) \text{ outer },|f_o| \ge 1 \text{ a.e.n. }E\}\\
&\geq\inf\{\|f_o\wedge 1\|^2_{\mathcal{D}_\mu} : f_o\in \mathcal{H}(b),|f_o| \ge 1 \text{ a.e.n. }E\}\\
&= \inf\{\|f_o\wedge 1\|^2_{\mathcal{D}_\mu} : f_o\in \mathcal{H}(b),|f_o\wedge 1| = 1 \text{ a.e.n. }E\}\\
&\geq \inf\{\|g\|^2_{\mathcal{D}_\mu} : g\in \mathcal{H}(b),|g| = 1 \text{ a.e.n. }E\},
\end{align*}
where in the last inequality we have used the inclusion between the two sets involved in the infima.
We recognize the expression in \eqref{Expr1}.
\end{proof}

By applying Proposition \ref{P:cyclicity} to Theorem 3.6 of \cite{FricainGrivaux}, we are able to prove the following theorem.
\begin{theorem}\label{zeta}
Let $b \in H^\infty$ with $\|b\|_{H^\infty}=1$ and $\mu$ a finite measure on $\mathbb{T}$. Assume that $\mathcal{H}(b)=\mathcal{D}_\mu$ and $\zeta \in \mathbb{T}$. The following properties are equivalent:
\begin{itemize}
\item[i)] $k^b_\zeta$ belongs to $\mathcal{H}(b)$.
\item[ii)] The polynomial $z-\zeta$ is not cyclic in $\mathcal{H}(b)$.
\item[iii)] $c_{b,1}(\zeta)>0$.
\item[iv)] $c_{b,2}(\zeta)>0$.
\end{itemize}
\end{theorem}
\begin{proof}
For the proof, see Lemma 3.2 of \cite{ELFALLAH20163262}.
\end{proof}

\begin{proof}[Proof of Theorem \ref{T:main3}]
We note that a function $f \in \mathcal{H}(b)$ is cyclic if and only if it is cyclic in $\mathcal{D}_\mu$ as well. Therefore, we apply Theorem 1 of \cite{ELFALLAH20163262}. We note that due to Proposition \ref{P:cyclicity}, as definition of $\mathcal{H}(b)$-capacity, we can consider both $c_{b,1}$ and $c_{b,2}$.
\end{proof}

\section*{Acknowledgements}

The first two authors are members of Gruppo Nazionale per l'Analisi Matematica, la Probabilità e le loro Applicazioni (GNAMPA) of Istituto Nazionale di Alta Matematica (INdAM). The first author was partially supported by PID2021-123405NB-I00 by the Ministerio de Ciencia e Innovación. The third author was partially supported by the exchange programme FRQ-CRM-CNRS and by ANR-24-CE40-5470. The fourth author was supported by the Discovery grant of NSERC and the Canada Research Chairs program.

\bibliographystyle{plain}
\bibliography{mybibliography}

\begin{thebibliography}{10}

\bibitem{Aleksandrov2000}
Alexei~B. Aleksandrov.
\newblock On embedding theorems for coinvariant subspaces of the shift
  operator. {II}.
\newblock {\em Journal of Mathematical Sciences}, 110:2907--2929, 2000.

\bibitem{Aleman1995}
Alexandru Aleman and Aristomenis~G. Siskakis.
\newblock An integral operator on {$H^p$}.
\newblock {\em Complex Variables, Theory and Application: An International
  Journal}, 28(2):149--158, October 1995.

\bibitem{BARANOV2005116}
Anton~D. Baranov.
\newblock {B}ernstein-type inequalities for shift-coinvariant subspaces and
  their applications to {C}arleson embeddings.
\newblock {\em Journal of Functional Analysis}, 223(1):116--146, 2005.

\bibitem{Baranovstability}
Anton~D. Baranov.
\newblock Stability of the bases and frames reproducing kernels in model
  spaces.
\newblock {\em Annales de l'Institut Fourier}, 55(7):2399--2422, 2005.

\bibitem{baranov2011feichtinger}
Anton~D. Baranov and Konstantin Dyakonov.
\newblock The {F}eichtinger conjecture for reproducing kernels in model
  subspaces.
\newblock {\em Journal of Geometric Analysis}, 21(2):276--287, 2011.

\bibitem{bellavita2023embedding}
Carlo Bellavita and Eugenio Dellepiane.
\newblock Embedding model and de {B}ranges-{R}ovnyak spaces in {D}irichlet
  spaces.
\newblock {\em Complex Variables and Elliptic Equations}, 70(2):228--246, 2025.

\bibitem{bellavita2024spectralanalysisdifferencequotient}
Carlo Bellavita, Eugenio Dellepiane, and Javad Mashreghi.
\newblock The spectral analysis of the difference quotient operator on model
  spaces.
\newblock {\em Bulletin des Sciences Mathématiques}, 204:103673, 2025.

\bibitem{Bellavita2022O}
Carlo Bellavita and Artur Nicolau.
\newblock One component bounded functions.
\newblock {\em Computational Methods and Function Theory}, 24(1):121--148,
  2024.

\bibitem{BELLAVITA2025110708}
Carlo Bellavita and Marco~M. Peloso.
\newblock {Duality, BMO and Hankel operators on Bernstein spaces}.
\newblock {\em Journal of Functional Analysis}, 288(2):110708, 2025.

\bibitem{Bergman2023OnCI}
Alex Bergman.
\newblock On cyclicity in de {B}ranges-{R}ovnyak spaces.
\newblock {\em Indiana University Mathematics Journal}, 73(4):1307--1329, 2024.

\bibitem{bessonov2015duality}
Roman~V. Bessonov.
\newblock Duality theorems for coinvariant subspaces of ${H^1}$.
\newblock {\em Advances in Mathematics}, 271:62--90, 2015.

\bibitem{bourdon}
Paul~S. Bourdon, Joseph~A. Cima, and Alec~L. Matheson.
\newblock Compact composition operators on {BMOA}.
\newblock {\em Transactions of the American Mathematical Society},
  351(6):2183--2196, 1999.

\bibitem{Brown1984CyclicVI}
Leon Brown and Allen~L. Shields.
\newblock Cyclic vectors in the {D}irichlet space.
\newblock {\em Transactions of the American Mathematical Society},
  285:269--303, 1984.

\bibitem{Chevrot2010}
Nicolas Chevrot, Dominique Guillot, and Thomas Ransford.
\newblock De {B}ranges–{R}ovnyak spaces and {D}irichlet spaces.
\newblock {\em Journal of Functional Analysis}, 259:2366--2383, 2010.

\bibitem{cima2006cauchy}
Joseph~A. Cima, Alec~L. Matheson, and William~T. Ross.
\newblock {\em The Cauchy Transform}.
\newblock Mathematical surveys and monographs. American Mathematical Society,
  2006.

\bibitem{cimamortini}
Joseph~A. Cima and Raymond Mortini.
\newblock One-component inner functions.
\newblock {\em Complex Analysis and its Synergies}, 3, 2017.

\bibitem{cohn}
William~S. Cohn.
\newblock {Carleson measures for functions orthogonal to invariant subspaces.}
\newblock {\em Pacific Journal of Mathematics}, 103(2):347--364, 1982.

\bibitem{costara2013}
Constantin Costara and Thomas Ransford.
\newblock Which de {B}ranges–{R}ovnyak spaces are {D}irichlet spaces (and
  vice versa)?
\newblock {\em Journal of Functional Analysis}, 265(12):3204--3218, 2013.

\bibitem{Debrangesrovnyak}
Louis de~Branges and James Rovnyak.
\newblock {\em Square summable power series}.
\newblock Holt, Rinehart and Winston, New York-Toronto-London, 1966.

\bibitem{Eugenionext}
Eugenio Dellepiane, Marco~M. Peloso, and Anita Tabacco.
\newblock On the equality of de {B}ranges-{R}ovnyak and {D}irichlet spaces.
\newblock {\em {C}omplex {A}nalysis and {O}perator {T}heory}, 19(119), 2025.

\bibitem{ELFALLAH20163262}
Omar El-Fallah, Youssef Elmadani, and Karim Kellay.
\newblock Cyclicity and invariant subspaces in {D}irichlet spaces.
\newblock {\em Journal of Functional Analysis}, 270(9):3262--3279, 2016.

\bibitem{ElFallah2015DirichletSW}
Omar El-Fallah, Karim Kellay, Hubert Klaja, Javad Mashreghi, and Thomas
  Ransford.
\newblock {D}irichlet spaces with superharmonic weights and de
  {B}ranges–{R}ovnyak spaces.
\newblock {\em Complex Analysis and Operator Theory}, 10:97--107, 2015.

\bibitem{primer-fkmr}
Omar El-Fallah, Karim Kellay, Javad Mashreghi, and Thomas Ransford.
\newblock {\em A primer on the {D}irichlet space}, volume 203 of {\em Cambridge
  Tracts in Mathematics}.
\newblock Cambridge University Press, Cambridge, 2014.

\bibitem{Eoff1995TheDN}
Carolyn Eoff.
\newblock The discrete nature of the {P}aley-{W}iener spaces.
\newblock {\em Proc. Amer. Math. Soc.}, 123(2):505--512, 1995.

\bibitem{FricainGrivaux}
Emmanuel Fricain and Sophie Grivaux.
\newblock Cyclicity in de {B}ranges–{R}ovnyak spaces.
\newblock {\em Moroccan Journal of Pure and Applied Analysis}, 9(2):216--237,
  2023.

\bibitem{Fricain_Hartmann_Ross_2018}
Emmanuel Fricain, Andreas Hartmann, and William~T. Ross.
\newblock Range spaces of co-analytic toeplitz operators.
\newblock {\em Canadian Journal of Mathematics}, 70(6):1261–1283, 2018.

\bibitem{Fricain2024CyclicityOT}
Emmanuel Fricain and Romain Lebreton.
\newblock Cyclicity of the shift operator through {B}ezout identities.
\newblock {\em Canadian Mathematical Bulletin}, 2024.

\bibitem{hb2}
Emmanuel Fricain and Javad Mashreghi.
\newblock {\em The Theory of H(b) Spaces}, volume~2 of {\em New Mathematical
  Monographs}.
\newblock Cambridge University Press, 2016.

\bibitem{Seco}
Emmanuel Fricain, Javad Mashreghi, and Daniel Seco.
\newblock {\em Cyclicity in non-extreme de {B}ranges-{R}ovnyak spaces}, volume
  638, pages 131--136.
\newblock American Mathematical Society, 2015.

\bibitem{model}
Stephan~R. Garcia, Javad Mashreghi, and William~T. Ross.
\newblock {\em Introduction to model spaces and their operators}, volume 148 of
  {\em Cambridge Studies in Advanced Mathematics}.
\newblock Cambridge University Press, Cambridge, 2016.

\bibitem{garnett}
John~B. Garnett.
\newblock {\em Bounded analytic functions}.
\newblock Graduate texts in mathematics. Springer, New York, revised 1. ed
  edition, 2007.

\bibitem{hartmann2025}
Andreas Hartmann and Giuseppe Lamberti.
\newblock Interpolation and random interpolation in de {B}ranges-{R}ovnyak
  spaces, 2025, https://arxiv.org/abs/2502.09094.

\bibitem{hedenmalm2012theory}
Haakan Hedenmalm, Boris Korenblum, and Kehe Zhu.
\newblock {\em {T}heory of {B}ergman Spaces}.
\newblock Graduate Texts in Mathematics. Springer New York, 2012.

\bibitem{LaceyII}
Michael~T. Lacey.
\newblock {Two-weight inequality for the Hilbert transform: A real variable
  characterization, II}.
\newblock {\em Duke Mathematical Journal}, 163(15):2821--2840, 2014.

\bibitem{Lacey2014}
Michael~T. Lacey, Eric~T. Sawyer, Chun-Yen Shen, and Ignacio Uriarte-Tuero.
\newblock Two-weight inequality for the {H}ilbert transform: A real variable
  characterization, {I}.
\newblock {\em Duke Mathematical Journal}, 163(15), December 2014.

\bibitem{Lanucha2017}
Bartosz Lanucha and Maria Nowak.
\newblock {D}e {B}ranges - {R}ovnyak spaces and generalized {D}irichlet spaces.
\newblock {\em Publicationes Mathematicae Debrecen}, 91(1–2):171--184, July
  2017.

\bibitem{luocauchy}
Shuaibing Luo.
\newblock Corona theorem for the {D}irichlet-type space.
\newblock {\em Journal of Geometric Analysis}, 32(74):1--20, 2022.

\bibitem{LGR}
Shuaibing Luo, Caixing Gu, and Stefan Richter.
\newblock Higher order local {D}irichlet integrals and de {B}ranges-{R}ovnyak
  spaces.
\newblock {\em Advances in Mathematics}, 385:107748, 2021.

\bibitem{malmanseco}
Bartosz Malman and Daniel Seco.
\newblock Embeddings into de {B}ranges-{R}ovnyak spaces.
\newblock {\em Studia Mathematica}, 278(2):173--194, 2024.

\bibitem{Marzo2012}
Jordi Marzo, Shahaf Nitzan, and Jan-Fredrik Olsen.
\newblock Sampling and interpolation in de branges spaces with doubling phase.
\newblock {\em Journal d’Analyse Mathématique}, 117(1):365–395, June 2012.

\bibitem{MR2500010}
Javad Mashreghi.
\newblock {\em Representation theorems in {H}ardy spaces}, volume~74 of {\em
  London Mathematical Society Student Texts}.
\newblock Cambridge University Press, Cambridge, 2009.

\bibitem{MR2986324}
Javad Mashreghi.
\newblock {\em Derivatives of inner functions}, volume~31 of {\em Fields
  Institute Monographs}.
\newblock Springer, New York; Fields Institute for Research in Mathematical
  Sciences, Toronto, ON, 2013.

\bibitem{MR3431581}
Javad Mashreghi and Thomas Ransford.
\newblock A {G}leason-{K}ahane-\.zelazko theorem for modules and applications
  to holomorphic function spaces.
\newblock {\em Bulletin of the London Mathematical Society}, 47(6):1014--1020,
  2015.

\bibitem{nicolau2021}
Artur Nicolau and Atte Reijonen.
\newblock A characterization of one-component inner functions.
\newblock {\em Bulletin of the London Mathematical Society}, 53(1):42--52,
  2021.

\bibitem{Ryan2025}
Ryan O'Loughlin.
\newblock Bounded {T}oeplitz products on the {H}ardy space, 2025,
  https://arxiv.org/abs/2501.14069.

\bibitem{Plancherel1936FonctionsEE}
Michel Plancherel and George P{\'o}lya.
\newblock Fonctions enti{\`e}res et int{\'e}grales de fourier multiples.
\newblock {\em Commentarii Mathematici Helvetici}, 10:110--163, 1936.

\bibitem{Poltoratski}
Alexei~G. Poltoratskij.
\newblock The boundary behavior of pseudocontinuable functions.
\newblock {\em St. Petersburg Mathematical Journal}, 5(2), 1993.

\bibitem{Pouliasis2024WeightedDS}
Stamatis Pouliasis.
\newblock Weighted dirichlet spaces that are de branges-rovnyak spaces with
  equivalent norms.
\newblock {\em Journal of Functional Analysis}, 288(3):110717, 2025.

\bibitem{Richter1991}
Stefan Richter.
\newblock A representation theorem for cyclic analytic two-isometries.
\newblock {\em Transactions of the American Mathematical Society},
  328(1):325--349, 1991.

\bibitem{RichterSundberg1991}
Stefan Richter and Carl Sundberg.
\newblock A formula for the local {D}irichlet integral.
\newblock {\em Michigan Mathematical Journal}, 38:355--379, 1991.

\bibitem{Ross2013LENSLO}
William~T. Ross.
\newblock Lens lectures on {A}leksandrov-{C}lark measures.
\newblock 2013, https://facultystaff.richmond.edu/\~{ }wross/PDF/Clark.pdf.

\bibitem{saksman2007}
Eero Saksman.
\newblock An elementary introduction to {C}lark measures.
\newblock In {\em Topics in complex analysis and operator theory}, page
  85–136, Spain, 2007. Universidad de M{\'a}laga.
\newblock Winter School in Complex Analysis and Operator Theory.

\bibitem{sarHb}
Donald Sarason.
\newblock {\em Sub-Hardy Hilbert Spaces in the Unit Disk}.
\newblock The University of Arkansas Lecture Notes in the Mathematical
  Sciences. Wiley, 1994.

\bibitem{Sarason1997LocalDS}
Donald Sarason.
\newblock Local {D}irichlet spaces as de {B}ranges-{R}ovnyak spaces.
\newblock {\em Proceedings of the American Mathematical Society},
  125(7):2133--2139, 1997.

\bibitem{TolsaMathAnn2001}
Xavier Tolsa.
\newblock {B}{M}{O}, ${H}^1$, and {C}alderón-{Z}ygmund operators for non
  doubling measures.
\newblock {\em Mathematische Annalen}, 319(1):89–149, January 2001.

\bibitem{Tolsa2001}
Xavier Tolsa.
\newblock A {T(1)} theorem for non-doubling measures with atoms.
\newblock {\em Proceedings of the London Mathematical Society},
  82(1):195–228, January 2001.

\bibitem{anucha2023DeBS}
Bartosz Łanucha, Małgorzata Michalska, Maria Nowak, and Andrzej Sołtysiak.
\newblock {D}e {B}ranges – {R}ovnyak spaces and local {D}irichlet spaces of
  higher order.
\newblock {\em Forum Mathematicum}, 36:275 -- 284, 2023.

\end{thebibliography}

\bigskip

\end{document}